\documentclass{compositio}
\usepackage{amsmath,xypic}\usepackage{times}

\renewcommand{\H}{\mathrm{H}}
\renewcommand{\S}{\mathbf{S}}
\renewcommand{\phi}{\varphi}

\newcommand{\s}{\mathcal}
\newcommand{\f}{\mathfrak}

\newcommand{\C}{\mathbb{C}}
\newcommand{\F}{\mathbb{F}}
\newcommand{\M}{\mathbf{M}}
\newcommand{\N}{\mathrm{N}}
\newcommand{\Q}{\mathbb{Q}}
\newcommand{\Z}{\mathbb{Z}}
\newcommand{\Fbar}[1]{\bar{\F}_{#1}}
\newcommand{\Qbar}{\bar{\Q}}

\newcommand{\rar}{\rightarrow}

\newcommand{\set}[2]{\{\,#1:#2\,\}}
\newcommand{\angles}[1]{\langle #1\rangle}
\newcommand{\floor}[1]{\lfloor #1\rfloor}

\newcommand{\cm}{\text{,}}
\newcommand{\pd}{\text{.}}

\newcommand{\llbracket}{[\mkern-2mu[}
\newcommand{\rrbracket}{]\mkern-2mu]}
\newcommand{\llparen}{(\mkern-3mu(}
\newcommand{\rrparen}{)\mkern-3mu)}
\newcommand{\dblbrackets}[1]{\llbracket #1 \rrbracket}
\newcommand{\dblparens}[1]{\llparen #1 \rrparen}

\let\div\undefined
\DeclareMathOperator{\div}{div}
\DeclareMathOperator{\End}{End}

\DeclareMathOperator{\Gal}{Gal}
\DeclareMathOperator{\GL}{GL}
\DeclareMathOperator{\SL}{SL}
\DeclareMathOperator{\PGL}{PGL}
\DeclareMathOperator{\PSL}{PSL}
\DeclareMathOperator{\PXL}{PXL}

\DeclareMathOperator{\dist}{dist}
\DeclareMathOperator{\ord}{ord}

\DeclareMathOperator{\im}{im}

\DeclareMathOperator{\Tate}{Tate}
\DeclareMathOperator{\Sturm}{Sturm}
\DeclareMathOperator{\Red}{Red}

\DeclareMathOperator{\modulo}{\,mod}

\newtheorem{theorem}{Theorem}[section] % number like 3.1, 3.2, 3.3, etc.
\newtheorem{proposition}[theorem]{Proposition} % number like 3.1, 3.2, 3.3, etc.
\newtheorem{lemma}[theorem]{Lemma} % number like 3.1, 3.2, 3.3, etc.
   \theoremstyle{definition} % Õhere we change the styleÕ
   \newtheorem{definition}[theorem]{Definition} % numbered with thm
   \theoremstyle{remark} % Õstyle changed againÕ
   \newtheorem{remark}{Remark}
   
   \newtheorem{conjecture}{Conjecture}
   
   \newtheorem{example}{Example}

\begin{document}

\title{Hecke stability and weight $1$ modular forms}

\author{George J. Schaeffer}
\email{gschaeff@math.ucla.edu}

\address{University of California, Los Angeles\\
Department of Mathematics\\
520 Portola Plaza\\
Math Sciences Building 6363\\
Mailcode: 155505}
\thanks{}

\begin{abstract}

The Galois representations associated to weight $1$ newforms over $\Fbar{p}$ are remarkable in that they are unramified at $p$, but the computation of weight $1$ modular forms has proven to be difficult. One complication in this setting is that a weight $1$ cusp form over $\Fbar{p}$ need not arise from reducing a weight $1$ cusp form over $\bar{\Q}$.

In this article we propose a unified {\it Hecke stability method} for computing spaces of weight $1$ modular forms of a given level in all characteristics simultaneously. Our main theorems outline conditions under which a finite-dimensional Hecke module of {\it ratios} of modular forms must consist of genuine modular forms.

We conclude with some applications of the Hecke stability method motivated by the refined inverse Galois problem.

\end{abstract}

\maketitle

%\section*{Notation}

\section*{Introduction and Motivation}

One of the major achievements of modern number theory is the discovery of a correspondence
\[\left\{\begin{array}{c}\text{newforms of weight $k$}\\\text{for $\Gamma_1(N)$ over $\Fbar{p}$}\end{array}\right\}\leftrightarrows\left\{\begin{array}{c}\text{representations $\Gal(\bar{\Q}/\Q)\rar\GL_2(\Fbar{p})$}\\\text{of ``Serre type'' unramified outside $Np$}\end{array}\right\}\]
established by the work of many researchers over the past few decades and codified in theorems of Eichler--Shimura \cite{Shi71}, Deligne \cite{Del68}, Deligne--Serre \cite{DS74}, Khare \cite{Kha06}, and Khare--Wintenberger \cite{KW09}.

Given a newform $f\in\S_k(N;\Fbar{p})$ as above one can construct a number field $K_f$ as follows: Let $\rho_f:\Gal(\bar{\Q}/\Q)\rar\GL_2(\Fbar{p})$ be the representation associated to $f$, and let $\overline{\rho}_f$ be the projectivization of $\rho_f$. Because $\rho_f$ is continuous, the image of $\overline{\rho}_f$ is a {\it finite} subgroup $G\le\PGL_2(\Fbar{p})$, and the fixed field $K_f$ of ${\ker\overline{\rho}_f}$ is a $G$-extension of $\Q$ that is unramified outside $Np$. 

This kind of construction is central to our understanding of the {\it refined inverse Galois problem} for finite subgroups of $\PGL_2(\Fbar{p})$ over $\Q$. It is particularly relevant to constructing $\mathrm{PXL}_2(\F_q)$-extensions of $\Q$ with limited ramification ($\PXL$ stands for ``$\PGL$ or $\PSL$'')---these Galois groups are nonsolvable when $q\ge 4$, so they are inaccessible to the methods of class field theory. 

The weight $1$ case of the correspondence above is unusual for several reasons:

\begin{itemize}\item[I.]{\bf A weight $1$ cusp form of level $N$ over $\Fbar{p}$ need not arise from reducing such a form over $\bar{\Q}$.}

When $p\nmid N$, the reduction map $\S_k(N;\Z[\tfrac{1}{N}])\rar\S_k(N;\F_p)$ is surjective provided that $k\ge 2$, but when $k=1$, surjectivity of reduction mod $p$ can fail for finitely many $p$ per level $N$. The first example of this phenomenon is due to Mestre at $(N,p)=(1429,2)$ \cite{Edi06}.

Surjectivity of $\S_1(N;\Z[\tfrac{1}{N}])\rar\S_1(N;\F_p)$ fails precisely when $\H^1(X_1(N),\underline{\omega}(-\mathrm{cusps}))$ has nontrivial $p$-torsion \cite{Kha07}. The data produced by our method suggest that the torsion of this cohomology group grows rapidly in $[\SL_2(\Z):\Gamma_1(N)]$ (see Section \ref{S4-growth}).

\begin{remark}Occasionally, even if a given $f\in\S_1(N;\F_p)$ does not come from a form in $\S_1(N;\Z[\frac{1}{N}])$, there {\it may} exist an {augmented level} $N'$ such that $f$ is the reduction of a form in $\S_1(N';\Z[\frac{1}{N}])$. This was observed by Buzzard at $(N,p)=(74,3)$ \cite{Buz12}. Mestre's prototype does not lift to a weight $1$ form in characteristic zero at any level.\end{remark}

\item[II.]{\bf Representations associated to weight $1$ newforms over $\Fbar{p}$ are unramified at $p$.}

Accordingly, the number field $K_f$ is unramified at $p$ when $f$ is of weight $1$. This is a theorem of Coleman and Voloch \cite{CV92} with a condition at $p=2$; the condition at $p=2$ was relaxed by Wiese \cite{Wiese2}.

\item[III.]{\bf Such representations potentially have ``large image.''}

Let $f\in\S_1(N;\Fbar{p})$ be a newform and let $\rho_f$ be the associated Galois representation. If $f$ lifts to a newform $F\in\S_1(N';\bar{\Q})$ for some level $N'$ augmenting $N$, then $\rho_f$ lifts to an Artin representation $\rho_{F}:\Gal(\bar{\Q}/\Q)\rar\GL_2(\C)$ of conductor $N'$ by the theorem of Deligne--Serre \cite{DS74}. Thus, $\Gal(K_f/\Q)$ is isomorphic to a finite subgroup of $\PGL_2(\C)$, and in particular, $\Gal(K_f/\Q)$ is solvable unless it is isomorphic to $\mathrm{A}_5$ \cite{Dickson}.

On the other hand, if the newform $f$ does {\it not} lift to a weight $1$ for in characteristic zero at any level, then the image of $\overline{\rho}_f$ is not necessarily isomorphic to a finite subgroup of $\PGL_2(\C)$. Given the $q$-expansion of $f$ to reasonably high precision, one can verify rather quickly that $\Gal(K_f/\Q)$ contains a copy of $\PSL_2(\F_q)$ (as in \cite{Serre} or \cite{Buz12}); if $q\ge4$ then $K_f$ must be a Galois number field with a nonsolvable Galois group, ramified only at primes dividing $N$ (by II.).

Mestre's example yields a $\PGL_2(\F_8)$-extension of $\Q$ ramified only at $1429$. In Section \ref{S4-gross} we will show that there exists a $\PGL_2(\F_{74873})$-extension of $\Q$ ramified only at $7$.

\item[IV.]{\bf Current methods for computing $\S_1(N;\Fbar{p})$ depend on $p$.}

To {\it compute} a space $V$ of modular forms over a ring $R$ is to give an algorithm that produces on input $P\in\Z_{\ge 0}$ a generating set for the image of the {\it truncated $q$-expansion map} $V\rar R\dblbrackets{q}/(q^P)$ (at some fixed cusp). Because of (III.) and because there are already algorithms for computing forms of higher weights \cite{Ste05}, our main focus in this paper is computing spaces of the form $\S_1(N,\chi;F)$ where $F=\bar{\Q}$ or $F=\Fbar{p}$ with $p\nmid N$, and $\chi$ is an odd Dirichlet character of level $N$ taking values in $F$.

As observed by Edixhoven, one can compute $\S_1(N,\chi;\Fbar{p})$ {\it for a given value of $p$} via the exact sequence
\[0\rightarrow\S_1(N,\chi;\Fbar{p})\xrightarrow{f(q)\mapsto f(q^p)}\S_p(N,\chi;\Fbar{p})\xrightarrow{f(q)\mapsto qf'(q)}\S_{p+2}(N,\chi;\Fbar{p})\]
(see \cite{Edi06}), but the complexity of this depends on the choice of $p$---one must compute the auxiliary space $\S_p(N,\chi;\Fbar{p})$ of dimension roughly $\frac{1}{12}p^2$. For this reason and (I.) we are motivated to formulate a method for computing weight $1$ modular forms that is ``characteristic-free.''

\begin{remark}It should also be mentioned that we do not have general dimension formulas for spaces of weight $1$ cusp forms, even over $\C$.\end{remark}

\end{itemize}

In \cite{Edi06}, Edixhoven notes that
\begin{quote}``... There seem to be no tables of mod $p$ modular forms of weight one, and worse, no published algorithm to compute such tables.''\end{quote}

The goal of this article is the development of the Hecke stability method (HSM), a procedure for computing and analyzing weight 1 modular forms in a way that takes the issues above into account. The HSM addresses and---in most cases---solves the problem posed by Edixhoven.

\subsection*{Outline}

In Section \ref{S1} we outline the central ideas of the Hecke stability method and state our main results. Section \ref{S2} contains an analysis of isogeny graphs and the proofs of the Hecke stability theorems (Theorems \ref{HSThm1} and \ref{HSThm2}). In Section \ref{S3} we prove Theorem \ref{DetectionThm} by explaining the practicalities of computing weight $1$ modular forms using Hecke stability.

Finally, in Section \ref{S4} we give some conjectures, examples, and applications that arise from the Hecke stability method.

\subsection*{Acknowledgements}

The author would like to thank Akshay Venkatesh, John Voight, Kevin Buzzard, Frank Calegari, Chandrashekhar Khare, and David Roberts for their input and support.

\section{Hecke stability and main results}\label{S1}

%\noindent{\it Notation.} 

\subsection{Hecke stability in general}\label{S1-general}

%We begin by describing the idea of Hecke stability in a general context.

For a fixed level $N\ge 1$ and a field $F$ in which $N$ is nonzero, we construct the $F$-algebra $\M(N;F)=\bigoplus_{k\ge 0}\M_k(N;F)$ of modular forms for $\Gamma_1(N)$ graded by weight. The space of {\it modular ratios for $\Gamma_1(N)$ over $F$}, denoted $\M^*(N;F)$, is the $\Z$-graded $F$-algebra generated by ratios of homogeneous elements from $\M(N;F)$. We treat $\M(N;F)$ as a subalgebra of $\M^*(N;F)$ in the obvious fashion.

As one might expect, much of the theory of modular forms carries over to the setting of modular ratios (see Section \ref{S2-ratios}). In particular, the action of the Hecke algebra $\mathbb{T}(N;F)$ extends to $\M^*(N;F)$ in a way that is compatible with the grading by weight, and $\M(N;F)$ is a Hecke submodule of $\M^*(N;F)$.

The Hecke stability method depends on characterizing the {\it finite-dimensional} Hecke-stable subspaces $V\subseteq\M^*(N;F)$. The idea is that such spaces {\it ought} to consist of modular forms, but because of certain complications on the supersingular locus of $X_1(N)$ (see Section \ref{S2-structure}) this is not entirely true.

\begin{theorem}\label{HSThm1}Let $N\ge 1$, let $F$ be a field in which $N$ is nonzero, and let $\ell$ be any prime that does not divide $N$ and that is not the characteristic of $F$. Suppose that $V$ is a finite-dimensional space of modular ratios and that the Hecke operator $T_\ell$ acts on $V$.
\begin{itemize}\item[a.]If $\mathrm{char}(F)=0$, then $V$ is a subspace of $\M(N;F)$.
\item[b.]If $\mathrm{char}(F)=p$ for $p>0$, there exists $r\ge 0$ such that $A^{r} V\subseteq \M(N;F)$ where $A$ is the characteristic $p$ Hasse invariant. In other words, if $f\in V$ and $f(\tau)=\infty$ for some $\tau\in X_1(N)(F)$, then $\tau$ is supersingular.
\end{itemize}\end{theorem}

%\begin{remark}A fortiori, the above conclusions also apply to any finite-dimensional $V\subseteq \M^*(N;F)$ stable under the action of the whole Hecke algebra $\mathbb{T}(N;F)$.\end{remark}

%The above implies that all finite-dimensional Hecke submodules of $\M^*(N;F)$ consist of modular forms, and 

%Because the $q$-expansion of the Hasse invariant $A$ is equal to $1$ at each cusp,
%
%\begin{corollary}Let $f\in\M^*(N;F)$. The following are equivalent:
%\begin{itemize}\item[i.]There is $g\in\M(N;F)$ such that $f(q)=g(q)$ at some (any) cusp of $X_1(N)$, and
%\item[ii.]The smallest Hecke submodule of $\M^*(N;F)$ containing $f$ is finite-dimensional.
%\end{itemize}\end{corollary}

To make the theorem above a practical tool for computing modular forms, we need conditions under which the exponent $r$ in claim (b.) is zero.

\subsection{Hecke stability and computing weight $1$ modular forms}\label{S1-wt1}

In light of (II.) and (III.) from the introduction, we are principally interested in computing spaces of the form $\S_1(N,\chi;F)$ where $F=\bar{\Q}$ or $F=\Fbar{p}$ and $\chi:(\Z/N\Z)^\times\rar F^\times$ is an odd character. Because of (IV.), we want our method to depend as little as possible on the choice of $F$.

%\footnote{idk how I feel about this paragraph}Because $F$ is algebraically closed, the space $\M_1(N;F)$ of modular forms for $\Gamma_1(N)$ over $F$ breaks up as a direct sum of $\mathbb{T}(N;F)$-modules $\bigoplus_\chi\M_1(N,\chi;F)$ where the {\it nebentypus} character $\chi$ ranges over all odd characters $(\Z/N\Z)^\times\rar F^\times$. As in the higher weight cases, each component $\S_1(N,\chi;F)$ of the cuspidal subspace is computed separately.

The {\it Hecke stability method (HSM)} for computing $\S_1(N,\chi;F)$ proceeds as follows: Fix a finite nonempty $\Lambda\subseteq\M_1(N,\chi^{-1};F)-\{0\}$ whose elements can be easily computed; $\Lambda$ could for example consist of explicit weight $1$ Eisenstein series. For each $\lambda\in\Lambda$ there is an injective map
\[[\lambda^{-1}]:\S_2(N,\boldsymbol{1};F)\rar\M_1^*(N,\chi;F):g\mapsto g/\lambda\text{,}\]
where $\M_1^*(N,\chi;F)$ is a space of modular ratios with modularity properties like the modular forms in $\M_1(N,\chi;F)$; in particular, $\M_1(N,\chi;F)=\M(N;F)\cap\M_1^*(N,\chi;F)$. Let $V_\Lambda'(F)=\bigcap_{\lambda\in\Lambda}\im[\lambda^{-1}]$. This is a finite-dimensional subspace of $\M_1^*(N,\chi;F)$ containing $\S_1(N,\chi;F)$. Elements of $V'_\Lambda(F)$ can have poles, but these are limited to the (finite) set \[\mathrm{Z}(\Lambda)=\set{\tau\in X_1(N)(F)}{\text{$\lambda(\tau)=0$ for all $\lambda\in\Lambda$}}\pd\]

Next, fix a prime $\ell$ such that $\ell\nmid N$ and $\ell\ne\mathrm{char}(F)$. The maximal $T_\ell$-stable subspace of $V'_\Lambda(F)$, denoted by $V_{\Lambda,\ell}'(F)$, contains the target space $\S_1(N,\chi;F)$. By Theorem \ref{HSThm1}, if $\mathrm{Z}(\Lambda)$ contains no supersingular points of $X_1(N)(F)$ we have the desired inclusions
\[\S_1(N,\chi;F)\subseteq V_{\Lambda,\ell}'(F)\subseteq \M_1(N,\chi;F)\pd\]

Unfortunately, the assumption that $\mathrm{Z}(\Lambda)$ contains no supersingular points is rather strong and also somewhat expensive to verify in practice (see Remark \ref{Zeros-Remark}). Indeed, if the modular curve $X_0(N)$ has elliptic points, it fails automatically for $F=\Fbar{p}$ when $p\not\equiv 1\modulo 12$. The main focus of Section \ref{S2} is weakening this hypothesis, and the theorem below summarizes the results of our efforts.

\begin{theorem}\label{HSThm2}Let $N\ge 1$, let $F$ be a field in which $N$ is nonzero, and let $\chi:(\Z/N\Z)^\times\rar F^\times$ be an odd character. Fix a prime $\ell$ such that $\ell\nmid N$ and $\ell\ne\mathrm{char}(F)$, let $\Lambda$ be a nonempty finite subset of $\M_1(N,\chi^{-1};F)-\{0\}$, and define $V'_{\Lambda,\ell}(F)$ and $\mathrm{Z}(\Lambda)$ as above.

We have $V_{\Lambda,\ell}'(F)\subseteq\M_1(N,\chi;F)$ if any one of the following conditions holds:

\begin{itemize}\item[i.]The characteristic of $F$ is $0$.

\item[ii.]The characteristic of $F$ is $p>0$ and $\mathrm{Z}(\Lambda)$ contains no supersingular points of $X_1(N)(F)$.

\item[iii.]The characteristic of $F$ is $p>0$ and there exist $M\mid N$ and $r\ge 2$ satisfying $p>\max\{\frac{4}{M}\ell^{4r},4\ell^{2r}\}$ and
\[|\delta_{N,M}\mathrm{Z}(\Lambda)\cap\set{\tau\in X_0(M)(F)^{\mathrm{ss}}}{\text{$\tau$ not elliptic}}|<\ell^{\floor{r/2}}+\ell^{\floor{r/2}-1}\cm\]
where $\delta_{N,M}:X_1(N)(F)\rar X_0(M)(F)$ is a modular degeneracy map (see Section \ref{S2-graphs}).

%\item[iv.]The characteristic of $F$ is $p>0$, the set $\delta_{N,1}\mathrm{Z}(\Lambda)\cap X(1)(F)^\mathrm{ss}$ consists of a single elliptic point $\tau'$, and there is a non-elliptic $\tau\in X(1)(F)$ such that $\Phi_\ell(X,Y)$ has a simple zero at $(j(\tau),j(\tau'))$.
%
%\item[v.]The characteristic of $F$ is $p>0$, The set $\delta_{N,1}\mathrm{Z}(\Lambda)\cap X(1)(F)^\mathrm{ss}$, consists of two elliptic points, $p>4\ell^4$, and there is a non-elliptic $\tau\in X(1)(F)$ and an elliptic $\tau'\in X(1)(F)$ such that $\Phi_\ell(X,Y)$ has a simple zero at $(j(\tau),j(\tau'))$.

\end{itemize}\end{theorem}

It is important to note that the space $V'_{\Lambda,\ell}(F)$ is easy to compute and that the auxiliary computations involved do not depend in a crucial way on the choice of $F$:
\begin{itemize}\item A basis for $\S_2(N,\boldsymbol{1};\Z[\frac{1}{N}])$ can be computed to arbitrarily high precision by several methods (modular symbols \cite{Ste05}, the method of graphs \cite{Mes11}, etc.). Because reduction mod $p$ is surjective for weight $2$ cusp forms, this also yields bases for $\S_2(N,\boldsymbol{1};\F_p)$ for all $p\nmid N$.

\item A standard choice for $\Lambda$ is a finite set of weight $1$ Eisenstein series for (the primitive of) the character $\chi^{-1}$. Such series can be computed directly as $q$-expansions over an extension of $\Z[\chi]$ inverting finitely many (explicitly computable) primes \cite{DS05}. 

\item The action of $T_\ell$ on modular ratios is integral and can be interpreted on $q$-expansions (see Section \ref{S2-ratios}).\end{itemize}

The computation of $V'_{\Lambda,\ell}(F)$ therefore amounts to computing $\S_2(N,\boldsymbol{1};\Z[\frac{1}{N}])$ (once per level), a suitable $\Lambda$ (essentially once per character), and performing some linear algebra in $F\dblparens{q}$. Of course, in actual implementations, we work in $F\dblparens{q}/(q^P)$ for $P$ sufficiently large; a lower bound on the precision $P$ required to unequivocally compute $V'_{\Lambda,\ell}(F)$ is in $\mathrm{O}(\ell^2N)$ and this bound does not depend on $F$ (see Lemma \ref{Lemma-Precision}).

\begin{remark}The perhaps unusual bounds in condition (iii.) of Theorem \ref{HSThm1} are stated in order to be as general as possible; there are of course many specific situations in which these bounds can be significantly improved (see Example \ref{Elliptic-Improvements} and Remark \ref{CM-Improvements}). 

If we assume {\it no a priori knowledge of $\mathrm{Z}(\Lambda)$}, then, taking $M$ to be the conductor of $\chi$, the requirements of (iii.) and a standard bound on $|\delta\mathrm{Z}(\Lambda)|$ coming from Riemann--Roch on $X_0(M)$ imply that the desired inclusion holds for all $p$ larger than some bound in $\mathrm{O}(M^7)$.

Even when such bounds do not apply in practice, one can almost always certify the hypothesis $V'_{\Lambda,\ell}(F)\subseteq\M_1(N,\chi;F)$ once $V'_{\Lambda,\ell}(F)$ has been computed (see Section \ref{S3-certification}).\end{remark}

\subsection{Detection of torsion cohomology, computing in all characteristics}\label{S1-detection}

%[[I kind of hate this section. We need to do something to it.]]

Because of the phenomenon described in (I.), we also require a procedure for listing those primes $p\nmid N$ for which the reduction map $\S_1(N;\Z[\frac{1}{N}])\rar\S_1(N;\F_p)$ is not surjective. As alluded to in the introduction, these are the primes such that $\H^1(X_1(N),\underline{\omega}(-\mathrm{cusps}))$ has nontrivial $p$-torsion where $\underline{\omega}$ is the sheaf of weight $1$ Katz modular forms for $\Gamma_1(N)$ \cite{Katz} \cite{Kha07}. It suffices to compute for each $\chi:(\Z/N\Z)^\times\rar\Qbar$ a list $L_\chi$ of prime ideals of $\Z[\frac{1}{N},\chi]$ containing all those $\f{p}\subseteq\Z[\frac{1}{N},\chi]$ for which $\S_1(N,\chi;\Z[\tfrac{1}{N},\chi])\rar\S_1(N,\chi;\s{O}_K/\f{p})$ is not surjective. 

For each character $\chi$ of level $N$, a full implementation of the Hecke stability method in all characteristics therefore requires two passes and a certification step:

\begin{itemize}\item In the first pass we compute $\M_1(N,\chi;\Qbar)$ using Hecke stability (as outlined in Section \ref{S1-wt1} and detailed in Section \ref{S3-HSM}) and Theorem \ref{HSThm2}. This also produces a finite list $L_\chi$ of primes of $\Q(\chi)$ containing all $\f{p}$ at which reduction is not surjective (see Section \ref{S3-detection}).
\item In the second pass, for each $\f{p}\in L_\chi$ we use Hecke stability again to compute a finite-dimensional $T_\ell$-stable $V'(\s{O}_K/\f{p})\subseteq\M_1^*(N,\chi;\s{O}_K/\f{p})$ containing $\S_1(N,\chi;\s{O}_K/\f{p})$.

\item Finally, for each $\f{p}\in L_\chi$, we must certify the {\it Hecke stability hypothesis} ``$V'(\s{O}_K/\f{p})\subseteq\M_1(N,\chi;\s{O}_K/\f{p})$'' by verifying one of the conditions of Theorem \ref{HSThm2} or by some alternate method (see Section \ref{S3-certification}).\end{itemize}

Putting all of our computational work together, we have the following theorem:

\begin{theorem}\label{DetectionThm}There is an algorithm that on input $(N,\chi,\ell,P)$ (with $N,\chi,\ell$ as above and $P\ge 0$) outputs the following:
\begin{itemize}\item A basis for $\M_1(N,\chi;\Qbar)$ computed to precision $P$;
\item A (finite) list $L_\chi$ of primes $\f{p}\subseteq\Z[\frac{1}{N},\chi]$ containing all those $\f{p}$ for which $\S_1(N,\chi;\Z[\tfrac{1}{N},\chi])\rar\S_1(N,\chi;\s{O}_K/\f{p})$ is not surjective;
\item For all $\f{p}\in L_\chi$ a basis for a space $V'(\s{O}_K/\f{p})$ that is $T_\ell$-stable and that satisfies
\[\S_1(N,\chi;\s{O}_K/\f{p})\subseteq V'(\s{O}_K/\f{p})\subseteq \M_1^*(N,\chi;\s{O}_K/\f{p})\text{; and}\]
\item Certificates that guarantee the inclusion $V'(\s{O}_K/\f{p})\subseteq\M_1(N,\chi;\s{O}_K/\f{p})$ for each $\f{p}\in L_\chi$ when certification is possible.
\end{itemize}\end{theorem}

\section{Proof of the Hecke stability theorems}\label{S2}

Throughout this section, fix a level $N\ge 1$ and an algebraically closed field $F$ such that $N$ is nonzero in $F$.

To prove Theorems \ref{HSThm1} and \ref{HSThm2} we will show that if $V\subseteq\M^*(N;F)$ is finite-dimensional and stable under the action of $T_\ell$, then there is a {\it lower bound} $B$ on the size of the set
\[\Pi(V)=\set{\tau\in X_1(N)(F)}{\text{there exists $f\in V$ with $f(\tau)=\infty$}}\cm\]
{\it provided that it is nonempty}. With notation as in Section \ref{S1-wt1}, we have $|\Pi(V_{\Lambda,\ell}'(F))|\le|\mathrm{Z}(\Lambda)|$. Therefore, if we can prove that $B>|\mathrm{Z}(\Lambda)|$ in a given situation, we would have $|\Pi(V_{\Lambda,\ell}'(F))|=0$, so $V_{\Lambda,\ell}'(F)$ would necessarily consist of modular forms.

%At the core of the proof is an analysis 

%At the core of our proof are isogeny graphs on the moduli of level structures. The Hecke operator $T_\ell$ is

%To formulate bounds on $\Pi(V)$ above, we will show that it is a {\it polar condition} on the isogeny graph in question. We then bound the size of polar conditions from below, provided that those graphs are sufficiently ``tree-like'' at a local level. The main difficulty is to show that the supersingular component of the $\ell$-isogeny graphs on $\Gamma_0$-structures over $F$ satisfies such a local condition on its structure.

%Our strategy is to formulate lower bounds on vertex sets called {\it polar conditions} that apply to graphs that are locally ``tree-like''; the set $\Pi(V)$ is a polar condition on the relevant isogeny graph. The main difficulty is then to prove that the local structural condition holds on the supersingular connected component of our graph.

\subsection{Modular ratios}\label{S2-ratios}

The space of modular ratios $\M^*(N;F)$ is the $\Z$-graded $F$-algebra generated by ratios of homogeneous elements from $\M(N;F)$. We denote the weight $k$ component of $\M^*(N;F)$ by $\M_k^*(N;F)$.

%For any $f\in\M^*(N;F)$ we say that a homogeneous $h\in\M(N;F)$ is a {\it denominator for $f$} if $hf$ is a modular form. If $\lambda\in\M_j(N;F)$ then there is an injection
%\[[\lambda^{-1}]:\M_{k}(N;F)\rar\M_{k-j}^*(N;F):g\mapsto g/\lambda\pd\]

Much of the theory of modular forms from \cite{Katz} applies to $\M^*(N;F)$:
\begin{itemize}\item Formally, a modular ratio over $F$ of weight $k$ is a global section of $\underline{\omega}_F^{k}\otimes\s{K}$ on the modular curve $X_1(N)$, where $\underline{\omega}$ is the sheaf of weight $1$ Katz modular forms on $X_1(N)$ and $\s{K}$ is the sheaf of rational functions on $X_1(N)$. 

\item At each cusp of $X_1(N)(F)$ there is a $q$-expansion map $\M^*(N;F)\rar F\dblparens{q^{1/N}}$ obtained by evaluating modular forms at the corresponding Tate object. In practice, we work with an implicit choice of cusp such that the image of $q$-expansion lies in $F\dblparens{q}$. If $f$ and $g$ are homogeneous modular forms, we have $(f/g)(q)=f(q)/g(q)$.

\item $\M^*(N;F)$ inherits the action of the diamond and Hecke operators on $\M(N;F)$. For each character $\chi:(\Z/N\Z)^\times\rar F^\times$ we denote by $\M^*(N,\chi;F)$ the subspace of $\M^*(N;F)$ on which $\angles{d}$ acts as multiplication by $\chi(d)$ for each $d\in(\Z/N\Z)^\times$.

\item If $\lambda\in\M_j(N,\theta;F)$, then there is an injection $[\lambda^{-1}]:\M_k(N,\chi;F)\rar\M_{k-j}(N,\chi\theta^{-1};F)$ that takes any $g\in\M_k(N,\chi;F)$ to $g/\lambda$.

%
%\item If $h\in\M_j(N;F)$ is a denominator for $f\in\M^*(N;F)$, then $Q_\ell h$ is a denominator for $T_\ell h$ where $Q_\ell$ is the multiplicative Hecke operator (...)

%\item MULTIPLICATIVE HECKE OPERATORx

\item If $f\in\M_k^*(N,\chi;F)$ and $\ell$ is a prime, then the $q$-expansion of $T_\ell f$ (at any cusp) satisfies
\[(T_\ell f)(q)=\sum_{n\in\Z}a_{\ell n/N}(f)q^{n/N}+\chi(\ell)\ell^{k-1}\sum_{n\in\Z}a_{n/N}(f)q^{\ell n/N}\]
where $a_r(f)$ is the coefficient of $q^r$ in the $q$-expansion of $f$ (at that same cusp). The proof of this formula is identical to the version for modular forms found in \cite{Katz}.
\end{itemize}

\begin{lemma}\label{Lemma-No-Cusps}Suppose that $f\in\M_k^*(N,\chi;F)$ for some weight $k$ and character $\chi$, and that the prime $\ell$ satisfies $\ell\nmid N$ and $\ell\ne\mathrm{char}(F)$.
\begin{itemize}\item[a.]If $\tau\in X_1(N)(F)$ is a cusp and $\ord_\tau(f)<0$, then $\ord_\tau(T_\ell f)=\ell\ord_\tau(f)$
\item[b.]If $V\subseteq\M^*(N;F)$ is finite-dimensional and stable under the action of $T_\ell$, then
\[\Pi(V)=\set{\tau\in X_1(N)(F)}{\text{there is $f\in V$ such that $f(\tau)=\infty$}}\]
contains no cusps.\end{itemize}\end{lemma}

\begin{proof}Claim (a.) follows from the $q$-expansion formula above since $\chi(\ell)\ell^{k-1}\ne 0$. From (a.) we see that if $\ord_\tau(f)<0$ then $\{\ord_\tau(T_\ell^nf)\}_{n\ge 0}$ is unbounded below, so any $T_\ell$-stable subspace containing $f$ must be infinite-dimensional; this proves (b.).\end{proof}

\subsection{Isogeny graphs and Hecke operators}\label{S2-graphs}

The $F$-points on the modular curve $Y_1(N)=X_1(N)-(\mathrm{cusps})$ represent isomorphism classes of $\Gamma_1(N)$-structures over $F$, and the Hecke operator $T_\ell$ encodes an {\it isogeny graph} on these points.

A {\it $\Gamma_1(N)$-structure} over $F$ is a pair $(E,P)$ where $E/F$ is an elliptic curve and $P\in E(F)$ has order $N$. A {\it $\Gamma_0(N)$-structure} over $F$ is a pair $(E,C)$ where $E/F$ is an elliptic curve and $C$ is a cyclic subgroup of $E(F)$ satisfying $|C|=N$. An {\it isogeny} of $\Gamma_i(N)$-structures $\phi:(E_1,\sigma_1)\rar(E_2,\sigma_2)$ is an isogeny $\phi:E_1\rar E_2$ of elliptic curves such that $\phi(\sigma_1)=\sigma_2$.

We consider two isogenies $\phi$ and $\psi$ to be {\it isomorphic} if there are isomorphisms $\iota_1$ and $\iota_2$ of $\Gamma_i(N)$-structures that make
\[\xymatrix{(E_1,\sigma_1)\ar[r]^\phi\ar[d]_{\iota_1}&(E_2,\sigma_2)\ar[d]^{\iota_2}\\(E_1',\sigma_1')\ar[r]_\psi&(E_2',\sigma_2')}\]
commute. We say that the isogenies $\phi$ and $\psi$ are {\it homotopic} if either $\phi\simeq\psi$ or $\widehat{\phi}\simeq\widehat{\psi}$ (where $\widehat{\phi}$ is the isogeny {\it dual} to $\phi$); the homotopy class of $\phi$ will be denoted $[\phi]$. Our justification for distinguishing between isomorphism and homotopy of isogenies will become clear later on.

Broadly speaking, an {\it isogeny graph} is a graph whose vertices are isomorphism classes of level structures and whose arcs are equivalence classes of isogenies between them. Fix $M\mid N$ and let $\ell$ be a prime that does not divide $\ell$ and that is not the characteristic of $F$. We will be working with the following three isogeny graphs:

{\small\begin{center}\begin{tabular}{|lll|}\hline  & vertices & arcs\\\hline

$\s{G}_\ell(\Gamma_1(N);F)$ & isomorphism classes of $\Gamma_1(N)$-structures & isomorphism classes of isogenies\\

$\s{G}_\ell(\Gamma_0(M);F)$ & isomorphism classes of $\Gamma_0(M)$-structures & isomorphism classes of isogenies\\

$\s{G}_\ell'(\Gamma_0(M);F)$ & isomorphism classes of $\Gamma_0(M)$-structures & homotopy classes of isogenies\\\hline

\end{tabular}\end{center}}

A priori, all of these graphs are directed and all of them may have loops and multiple arcs with the same origin and destination (technically, they are {\it directed pseudomultigraphs}). Because $F$ is algebraically closed, we may identify the vertex set of $\s{G}_\ell(\Gamma_1(N);F)$ with $Y_1(N)(F)$ and the vertex sets of $\s{G}_\ell(\Gamma_0(M);F)$ and $\s{G}'_\ell(\Gamma_0(M);F)$ with $Y_0(M)(F)$. Because of our assumptions on $\ell$, every vertex of $\s{G}_\ell(\Gamma_1(N);F)$ and $\s{G}_\ell(\Gamma_0(M);F)$ has outdegree $\ell+1$.

The adjacency relation of the isogeny graph $\s{G}_\ell(\Gamma_1(N);F)$ encodes the action of the Hecke operator $T_\ell$ on $\M^*(N;F)$ (see also \cite{Mes11}). If $f\in\M^*(N;F)$, $(E,P)$ is a $\Gamma_1(N)$-structure over $F$, and $\omega$ is a nonvanishing differential on $E$, then we can evaluate the modular ratio $T_\ell f$ on the {\it test object} $(E,P,\omega)$ by ``averaging'' $f$ at each $\ell$-isogenous test object:
\[(T_\ell f)(E,P,\omega)=\frac{1}{\ell}\sum_{\substack{H\le E[\ell]\\|H|=\ell}}f(E/H,\phi_H(P),\phi_{H*}\omega)\]
where $H$ ranges over the (cyclic) subgroups of $E(F)$ having order $\ell$ and $\phi_H:E\rar E/H$ is the quotient isogeny \cite{Katz}. In particular, if
\[\sum_{\substack{\tau'\in Y_1(N)(F)\\f(\tau')=\infty}}\deg_{\s{G}_\ell(\Gamma_1(N);F)}(\tau,\tau')=1\cm\]
then $(T_\ell f)(\tau)=\infty$. This motivates the following definition:

%More formally, recall that the $\ell$th Hecke correspondence on $Y_1(N)$ is defined by $\gamma_{2*}\gamma_1^*$ where
%\[\]
%Above, $Y_1(N;\ell)$ is the affine modular curve whose $F$-points are isomorphism classes of triples $(E,P,H)$ where $(E,P)$ is a $\Gamma_1(N)$-structure over $F$ and $H\le E(F)$ satisfies $|H|=\ell$. The morphisms $\gamma_1$ and $\gamma_2$ are interpreted on level structures by $(E,P,H)\mapsto (E,P)$ and $(E,P,H)\mapsto (E/H,\phi_H(P))$, respectively. The Hecke operator $T_\ell:\M^*(N;F)\rar\M^*(N;F)$ is the pullback of $\gamma_{2*}\gamma_1^*$, and the above claim follows from the fact that $\gamma_1$ and $\gamma_2$ are \'etale.

\begin{definition}\label{Definition-Polar}Let $G$ be a directed graph. A subset $P\subseteq G$ of vertices is called a {\it polar condition} on $G$ if for all vertices $v\in G$,
\[\deg(v,P)=\sum_{w\in P}\deg(v,w)=1\]
guarantees $v\in P$.\end{definition}

By the observation preceding Definition \ref{Definition-Polar},

\begin{proposition}\label{Proposition-Pi-Polar}Let $V\subseteq\M^*(N;F)$ and let
\[\Pi(V)=\set{\tau\in X_1(N)(F)}{\text{there exists $f\in V$ with $f(\tau)=\infty$}}\pd\]

If $V$ is $T_\ell$-stable and $\Pi(V)$ is nonempty, then $\Pi(V)\subseteq Y_1(N)$ (by Proposition \ref{Lemma-No-Cusps}) and $\Pi(V)$ is a polar condition on $\s{G}_\ell(\Gamma_1(N);F)$.\end{proposition}

Ultimately we want to bound $|\Pi(V)|$ from below using the fact that it is a polar condition, but working with $\Gamma_1(N)$-structures directly is awkward. Instead, we will pass to the isogeny graphs on $Y_0(M)(F)$ via the {\it modular degeneracy map} $\delta_{N,M}:Y_1(N)\rar Y_0(M)$ interpreted on level structures by $(E,P)\mapsto (E,\angles{\frac{N}{M}P})$. This degeneracy induces a surjective graph homomorphism $\s{G}_\ell(\Gamma_1(N);F)\rar\s{G}_\ell(\Gamma_0(M);F)$ by taking the arc $\tau\rar\tau'$ represented by some isogeny $\phi$ to the arc $\delta_{N,M}(\tau)\rar\delta_{N,M}(\tau')$ represented by that same isogeny. We want to show that the image of $\Pi(V)$ under $\delta_{N,M}$ is a polar condition on $\s{G}_\ell(\Gamma_0(M);F)$ (provided that it is nonempty).

\begin{lemma}\label{Lemma-Polar-Fibers}Let $r\ge 1$, let $G$ and $G'$ be directed graphs and let $P$ be a polar condition on $G$. Suppose that $\delta:G\rar G'$ is a surjective homomorphism of graphs such that
\begin{itemize}\item[i.] For all vertices $v\in G$ the map $\delta:\mathrm{arcs}_G(v,G)\rar\mathrm{arcs}_{G'}(\delta(v),G')$ is a bijection; and
\item[ii.] Adjacency in $G'$ ``lifts along fibers of $\delta$'': For all vertices $v',w'\in G'$,
\[\deg_{G'}(v',w')\ge 1\Rightarrow \forall w\in \delta^{-1}(w')\ \ \exists v\in \delta^{-1}(v')\ \ \deg_G(v,w)\ge 1\text{.}\]
That is, any diagram of the form
\[\begin{matrix}\xymatrix{&w\ar@{|->}[d]^\delta\\v'\ar[r]&w'}\end{matrix}\quad\text{can be completed to a diagram of the form}\quad\begin{matrix}\xymatrix{v\ar@{|-->}[d]_\delta\ar@{-->}[r]&w\ar@{|->}[d]^\delta\\v'\ar[r]&w'}\end{matrix}\pd\]

\end{itemize}
Then $\delta(P)$ is a polar condition on $G'$.\end{lemma}

\begin{proof}Let $v'\in G'$. If $\deg_{G'}(v',\delta(P))=1$, there exists $w'\in \delta(P)$ such that $\deg_{G'}(v',w')=1$. Fix $w\in \delta^{-1}(w')$ such that $w\in P$. By (ii.), there exists $v\in \delta^{-1}(v')$ such that $\deg_G(v,w)\ge 1$. Since $w\in P$, $\deg_G(v,P)\ge 1$.

On the other hand, because $\delta:\mathrm{arcs}_G(v,G)\rar\mathrm{arcs}_{G'}(v',G')$ is a bijection, $\delta:\mathrm{arcs}_G(v,P)\rar\mathrm{arcs}_{G'}(v',\delta(P))$ is an injection. Thus, $\deg_G(v,P)\le\deg_G(v',\delta(P))=1$. Since $P$ is a polar condition on $G$ and $\deg_G(v,P)=1$, we have $v\in P$. Therefore, $v'\in \delta(P)$, and this proves that $\delta(P)$ is also polar.\end{proof}

\begin{proposition}Let $V$ be a subspace of $\M^*(N;F)$ and define $\Pi(V)$ as in Proposition \ref{Proposition-Pi-Polar}.

If $V$ is stable under the action of $T_\ell$ and $\Pi(V)$ is nonempty, then $\delta_{N,M}\Pi(V)$ is a polar condition on $\s{G}_\ell(\Gamma_0(M);F)$.\end{proposition}

\begin{proof}We need only prove that $\delta_{N,M}:\s{G}_\ell(\Gamma_1(N);F)\rar\s{G}_\ell(\Gamma_0(M);F)$ satisfies conditions (i.) and (ii.) of Lemma \ref{Lemma-Polar-Fibers}. It has property (i.) by construction. To prove that it has property (ii.), it is sufficient to show that given an $\ell$-isogeny of $\Gamma_0(M)$-structures $\phi:(E_1,C_1)\rar (E_2,C_2)$ and $P_2\in E_2[N]$ with $C_2=\angles{\frac{N}{M}P_2}$, there exists $P_1\in E_1[N]$ such that 
\[\xymatrix{(E_1,P_1)\ar@{-->}[r]^\phi\ar@{|-->}[d]_{\delta_{N,M}}&(E_2,P_2)\ar@{|->}[d]^{\delta_{N,M}}\\(E_1,C_1)\ar[r]_\phi&(E_2,C_2)}\]
commutes. One easily verifies that $P_1=m\widehat{\phi}(P_2)$ where $\ell m\equiv 1\modulo N$ works.\end{proof}

A minor disadvantage of working with isogeny graphs on $\Gamma_0(M)$-structures is that $Y_0(M)(F)$ may contain (finitely many) elliptic points---points representing $\Gamma_0(M)$-structures over $F$ whose automorphism groups are strictly larger than $\{\pm 1\}$. If $(E,C)$ represents an elliptic point, then there could be $\phi,\psi:(E,C)\rar(E',C')$ such that $\phi\not\simeq\psi$ but $\widehat{\phi}\simeq\widehat{\psi}$. This means that $\phi$ and $\psi$ are {\it homotopic but not isomorphic}.

The isogeny graph $\s{G}_\ell'(\Gamma_0(M);F)$ (the last graph listed in the table above) is designed to circumvent this issue: Dualization of isogenies is a direction-reversing involution on the arc set of $\s{G}_\ell'(\Gamma_0(M);F)$, so we may consider $\s{G}_\ell'(\Gamma_0(M);F)$ as an {\it undirected graph} whose edges are orbits $\{[\phi],[\widehat{\phi}]\}$ under the action of dualization. It is necessary to distinguish between loops $[\phi]:\tau\rar\tau$ that are equal to their own duals ({\it self-dual} loops) from those that are not (see Remark \ref{Remark-Structure-Theorem}). Every nonelliptic vertex of $\s{G}_\ell'(\Gamma_0(M);F)$ has degree $\ell+1$ (note that a self-dual loop contributes $1$ to the degree of its vertex and a non-self-dual loop contributes $2$).

Let ${W}_M$ denote the set of elliptic points in $Y_0(M)(F)$.

\begin{lemma}\label{Lemma-Isogeny-Graph-Transfer}\begin{itemize}\item[a.]If $P$ is a polar condition on $\s{G}_\ell(\Gamma_0(M);F)$, then $P$ is a polar condition on $\s{G}_\ell'(\Gamma_0(M);F)$ as well.

\item[b.]If $P'$ is a polar condition on $\s{G}_\ell'(\Gamma_0(M);F)$, then there is a polar condition $P$ on $\s{G}_\ell(\Gamma_0(M);F)$ such that $P-{W}_M=P'-{W}_M$.
\end{itemize}
\end{lemma}

\begin{proof}Both of these follow from the fact that $\s{G}_\ell'(\Gamma_0(M);F)$ is a subgraph of $\s{G}_\ell(\Gamma_0(M);F)$, and the complement of this subgraph consists of (finitely many) arcs based at elliptic points of $Y_0(M)(F)$.\end{proof}

\subsection{Structure theory of isogeny graphs on $Y_0(M)$}\label{S2-structure}

To every $\tau\in Y_0(M)(F)$ we associate an abstract ring $\End(\tau)$ such that for every $\Gamma_0(M)$-structure $(E,C)$ representing $\tau$,
\[\End(\tau)\simeq\End(E,C)=\set{\alpha\in\End(E)}{\alpha(C)\subseteq C}\]
The ring $\End(\tau)$ is equipped with a {\it conjugation} operation $\alpha\mapsto\bar{\alpha}$ (corresponding to isogeny dualization) and a multiplicative {\it norm} $\N:\End(\tau)\rar\Z_{\ge 0}:\alpha\mapsto\alpha\bar{\alpha}$ (corresponding to isogeny degree). By the elementary classification of endomorphisms of elliptic curves, $\End(\tau)$ is isomorphic to $\Z$, an order of an imaginary quadratic number field, or an Eichler order of a quaternion algebra. If $\End(\tau)$ is commutative, we call $\tau$ {\it ordinary}, and we call $\tau$ {\it supersingular} otherwise.

%Our main strategy in this section will be to relate the combinatorial structure of $\s{G}_\ell(\Gamma_0(M);F)$ with the algebraic structure of $\End(\tau)$ as $\tau$ ranges over the vertices of this graph.

\begin{lemma}\label{Lemma-End-Transfer} Let $\tau,\tau'\in Y_0(M)(F)$. If $\alpha\in\End(\tau)$ and there is a homotopy class $[\phi]:\tau\rar\tau'$, then $\Z[d\alpha]\hookrightarrow\End(\tau')$ where $d=\deg\phi$.\end{lemma}

\begin{proof}Choose representatives $\phi:(E,C)\rar(E',C')$ for $[\phi]:\tau\rar\tau'$, let $\alpha\in\End(E,C)$, and consider the diagram
\[(E',C')\xrightarrow{\phi}(E,C)\xrightarrow{\alpha}(E,C)\xrightarrow{\widehat{\phi}}(E',C')\pd\]
Let $\beta=\widehat{\phi}\alpha\phi\in\End(E',C')$. The conclusion is trivial if $\alpha\in\Z$, so assume otherwise.

Fix any prime $\nu$ different from the characteristic of $F$. Applying the $\Tate_\nu$ functor to the diagram above and choosing bases for $\Tate_\nu(E)$ and $\Tate_\nu(E')$ allows us to identify the isogenies above with $2\times 2$ matrices over $\Z_\nu$.

Since $\beta\notin\Z$, the minimal polynomial of $\beta$ is equal to the characteristic polynomial of $\Tate_\nu(\widehat{\phi}\alpha\phi)$. The trace and determinant of a matrix product are invariant under cyclic permutations of the terms, so this is also equal to the characteristic polynomial of $\Tate_\nu(\phi\widehat{\phi}\alpha)=d\Tate_\nu(\alpha)$. It follows that $\beta$ and $d\alpha$ have the same minimal polynomial, so $\Z[d\alpha]\hookrightarrow\End(\tau')$, as claimed.\end{proof}

%In particular, if 
%
%%In particular, if $

Let $\tau\in Y_0(M)(F)$ and let $G$ be an undirected subgraph of $\s{G}'_\ell(\Gamma_0(M);F)$. The above lemma implies that if $d=\dist_{G}(\tau,\tau')$, then for all $\alpha\in\End(\tau)$, $\Z[\ell^d\alpha]\hookrightarrow\End(\tau')$. This hints at a structural relationship between the graph $G$ and the multiplicative monoid
\[\End_\ell(\tau)=\set{\alpha\in\End(\tau)}{\N(\alpha)\in\ell^{\Z_{\ge 0}}}\cm\]
a relationship we will develop and exploit heavily in what follows.

To simplify the next construction, fix representative $\Gamma_0(M)$-structures for each vertex of $G$ and choose compatible representative $\ell$-isogenies for each arc of $G$. We demand that if $\phi$ is chosen as a representative for an arc, then $\widehat{\phi}$ must be chosen as the representative of the dual arc, so that every edge has the form $\{\phi,\widehat{\phi}\}$. For the moment we will identify $G$ with the graph obtained via this choice of representatives. If
\[x:(E_1,C_1)\xrightarrow{\phi_1}(E_2,C_2)\rar\cdots\rar(E_n,C_n)\xrightarrow{\phi_n}(E_1,C_1)\]
is a directed cycle in $G$ based at $(E_1,C_1)$, we set
\[\xi_{(E_1,C_1)}(x)=\widehat{\phi}_1\cdots\widehat{\phi}_n\in\End(E_1,C_1)\text{.}\]
By convention, $\xi_{(E_1,C_1)}(x)=1$ if and only if $x$ is the trivial cycle based at $\tau$. We therefore obtain for each vertex $(E,C)$ a monoid homomorphism
\[\xi_{(E,C)}:\{\text{directed cycles in $G$ based at $(E,C)$}\}\rar\End(E,C)\]
where the operation on the left is concatenation (if $a$ and $b$ are paths with $a$ ending at the origin of $b$, then $ab$ denotes the path obtained by first following $a$ and then following $b$).

Forgetting our choices of representatives for the elements of $G$, we obtain a family of monoid homomorphisms $\{\xi_\tau\}_\tau$ where for each vertex $\tau$, $\xi_\tau$ is a map from directed cycles in $G$ based at $\tau$ to elements of the monoid $\End_\ell(\tau)$ defined above. Because of the many choices involved in the above construction, we explicitly avoid asserting any sort of canonicity for $\{\xi_\tau\}_\tau$.

\begin{theorem}\label{Theorem-Structure}Let $G$ be an undirected subgraph of $\s{G}_\ell(M;F)$. There is a family $\{\xi_\tau\}_\tau$ of monoid homomorphisms indexed by the vertices of $G$ such that
\[\xi_\tau:\{\text{directed cycles in $G$ based at $\tau$}\}\rar\End_\ell(\tau)\]
satisfying the following claims: If $x$ is a directed cycle in $G$ based at $\tau$, then
\begin{itemize}\item[a.]$\mathrm{N}(\xi_\tau(x))=\ell^{|x|}$ where $|x|$ is the length of $x$;
\item[b.]If $\xi_\tau(x)$ is irreducible then $x$ is irreducible;
\item[c.]For any other $\tau'$ lying on $x$, either $\xi_\tau(x)\in\Z$ and $\xi_\tau(x)=\pm\xi_{\tau'}(x)$, or $\xi_\tau(x)\notin\Z$ and there exists an irreducible quadratic polynomial $F\in\Z[t]$ such that $F(\xi_\tau(x))=F({\xi_{\tau'}(x)})=0$;
\item[d.]$\xi_\tau(x)\in\Z$ if and only if $x$ is contractible; and
\item[e.]Each $\xi_\tau$ induces an injective group homomorphism
\[\widetilde{\xi}_\tau:\pi_1(G,\tau)\rar\frac{\End_\ell(\tau)}{\Z\cap\End_\ell(\tau)}\text{.}\]
%\item[f.]If $\alpha\in\End_\ell(\tau)$ is irreducible, there exists (...)
%\item[g.]If $x$ is simple, 
%If $\tau$ is not isogenous (of any degree) to an elliptic point, then $\widetilde{\xi}_\tau$ is an isomorphism.
\end{itemize}
\end{theorem}

\begin{remark}\label{Remark-Structure-Theorem}We pause here for some remarks related to claim (e.) above. First of all, the monoid quotient given there is a group; inversion is induced by conjugation in $\End(\tau)$.

Secondly, self-dual loops in $G$ contribute $2$-torsion to $\pi(G,\tau)$: If $G$ is connected and $\s{T}$ is a spanning tree for $G$, then $G-\s{T}$ consists of edges (finitely many, in this setting). Let $s$ be the number of self-dual loops in $G-\s{T}$, and let $t$ be the number of all other edges. We have $\pi_1(G,\tau)\simeq (\Z/2\Z)^{*s}*\Z^{*t}$. Note that $\pi_1(G,\tau)$ contains an element of order $2$ only if $\End_\ell(\tau)$ contains a square root of $-\ell$.

Thirdly, if $G$ is a connected component and $G$ contains no elliptic points, then the group homomorphism $\widetilde{\xi}_\tau$ is an isomorphism for each $\tau\in G$.

\end{remark}

\begin{proof} Let $\{\xi_\tau\}_\tau$ be the family of monoid homomorphisms constructed before the statement of the theorem. Claim (a.) and the contrapositive of claim (b.) both follow directly from the construction.

Claim (c.) is proven by applying an appropriately chosen Tate functor to $x$ and remembering (as in the proof of Lemma \ref{Lemma-End-Transfer}) that the characteristic polynomial of a product of $2\times 2$ matrices is invariant under cyclic permutations of the terms.

Claim (d.) is proven by induction on $|x|$ in a series of if and only if statements. The case $|x|=0$ is trivial so assume $|x|>0$. The following are equivalent:
\begin{itemize}\item[i.]$x$ is contractible,
\item[ii.]There exist $\tau',\tau''\in Y_0(M)$ and an arc $[\phi]:\tau'\rar\tau''$ such that $x$ has the form
\[x:\tau\underbrace{\rightarrow\cdots\rightarrow}_a\tau'\xrightarrow{[\phi]}\tau''\xrightarrow{[\widehat{\phi}]}\tau'\underbrace{\rightarrow\cdots\rightarrow}_b\tau\]
where $y=ab$ is contractible, and
%\item[iii.]There exists a vertex $\tau_1$ on $x$ such that $\xi_{\tau_1}(x)=\ell\xi_{\tau_1}(y)$ where $y$ is contractible and $|y|<|x|$.
\item[iii.]$\xi_{\tau}(x)\in\Z$.\end{itemize}

(i.$\Leftrightarrow$ii.) follows from the construction of $\s{G}_\ell(M;F)$ as an undirected graph and a routine characterization of contractible cycles on an undirected graph. (ii.$\Rightarrow$iii.) follows from claim (c.), the fact that $\xi_{\tau'}([\phi][\widehat{\phi}])=\ell$, and the inductive hypothesis.

It remains to prove (iii.$\Rightarrow$ii.) If $\xi_{\tau}(x)\in\Z\cap\End_\ell(\tau)$ and $|x|>0$, then $\ell\mid\xi_{\tau}(x)$ in $\End_\ell(\tau)$. Fix representatives as in the discussion preceding the theorem and suppose that
\[x:(E_1,C_1)\xrightarrow{\psi_1}\cdots\rightarrow (E_n,C_n)\xrightarrow{\psi_n}(E_1,C_1)\]
where for each $i$, $(E_i,C_i)$ represents $\tau_i$, and $\tau_1=\tau$. The endomorphism $\xi_{(E_1,C_1)}(x)=\widehat{\psi}_1\cdots\widehat{\psi}_n$ of $(E_1,C_1)$ is divisible by $\ell$, so its kernel contains $E_1[\ell]$. For $k$ with $0\le k<n$, let $\eta_k=\widehat{\psi}_{n-k}\cdots\widehat{\psi}_n$. Since $\ker(\eta_0)$ is cyclic but $\ker(\eta_{n-1})$ is not, there is a least $k$ such that $\ker(\eta_k)$ is not cyclic. Then $\ker(\widehat{\psi}_{n-k}\widehat{\psi}_{n-k+1})=E_{n-k+2}[\ell]$, so $\widehat{\psi}_{n-k}\widehat{\psi}_{n-k+1}$ is multiplication by $\ell$ on the underlying elliptic curve of the $\Gamma_0(M)$-structure $(E_{n-k+2},C_{n-k+2})=(E_{n-k},C_{n-k})$. It follows from how we chose representatives that $\psi_{n-k+1}=\widehat{\psi}_{n-k}$. Taking isomorphism classes of vertices and homotopy classes of arcs yields vertices $\tau'=\tau_{n-k}$, $\tau''=\tau_{n-k+1}$, and an arc $[\phi]=[\psi_{n-k}]:\tau'\rar\tau''$ with the desired properties. Finally, to prove that the remainder cycle
\[y:\tau_1\xrightarrow{[{\psi}_1]}\cdots\xrightarrow{[{\psi}_{n-k-1}]}\tau_{n-k}\xrightarrow{[{\psi}_{n-k+2}]}\tau_{n-k+3}\rightarrow\cdots\rightarrow \tau_1\]
is contractible, note that $\xi_{\tau_1}(x)\in\Z$ implies $\xi_{\tau_{n-k}}(x)\in\Z$ by (c.), so since $\xi_{\tau_{n-k}}(x)=\ell\xi_{\tau_{n-k}}(y)$, it follows that $\xi_{\tau_{n-k}}(y)\in\Z$. Because $|y|<|x|$, $y$ is contractible by the inductive hypothesis.

Claim (e.) follows immediately from claim (d.).\end{proof}

If $G$ is a connected component of $\s{G}_\ell'(\Gamma_0(M);F)$ then either every vertex of $G$ is ordinary or every vertex of $G$ is supersingular. We may therefore distinguish between the ordinary components and supersingular components of $\s{G}_\ell'(\Gamma_0(M);F)$. 

\begin{itemize}\item By Theorem \ref{Theorem-Structure}.e, an ordinary component $G$ has at most one cycle, and the regularity of $G$ (away from elliptic points) implies that $G$ is infinite (it is either an infinite tree or an infinite volcano).

\item There is a supersingular component iff $\mathrm{char}(F)=p>0$, in which case the supersingular component is unique. In contrast with the ordinary components, the supersingular component of $\s{G}_\ell'(\Gamma_0(M);F)$ is finite, and its structure can be rather complicated.

\end{itemize}

%The ordinary components are relatively easy to describe since the endomorphism rings of ordinary elliptic curves are subrings of imaginary quadratic fields, hence commutative. By Theorem x, it follows that ordinary components of $\s{G}_\ell(\Gamma_0(M);F)$ have at most one cycle, and by regularity (away from elliptic points) such components are always infinite. By contrast, there is a supersingular component only if $\mathrm{char}(F)>0$, in which case it is unique and finite . The structure of the supersingular component is rather more difficult to describe completely, (...) possible applications in cryptography.

We will show that despite their apparent complexity, supersingular components of isogeny graphs resembles ordinary components ``locally.'' The key is the following lemma of Goren and Lauter.

\begin{lemma}[Goren--Lauter lemma]\label{Lemma-GL}Let $\tau\in Y_0(M)(F)$ and let $p=\mathrm{char}(F)>0$. If $\alpha,\beta\in\End(\tau)$ satisfy $\alpha\beta\ne\beta\alpha$, then $4\N(\alpha)\N(\beta)\ge Mp$.\end{lemma}

\begin{proof}If $\End(\tau)$ is not commutative, then $\tau$ is supersingular and $\End(\tau)$ is an Eichler order of level $M$ in a quaternion algebra $B$ ramified at $p$ and $\infty$; the discriminant of $\s{O}$ is $(Mp)^2$. The proof now proceeds as in \cite{GorenLauter} Section 2.1, which treats the case $M=1$.\end{proof}

For example, if $x$ and $y$ are directed cycles based at $\tau$ in the supersingular component $G$ of $\s{G}_\ell'(\Gamma_0(M);F)$, and the homotopy classes of $x$ and $y$ do not commute in $\pi(G,\tau)$, then combining the Goren--Lauter lemma with Theorem \ref{Theorem-Structure} yields $|x|+|y|\ge\log_\ell(\frac{Mp}{4})$. Colloquially, if $Mp$ is large compared to $\ell$, short cycles in the graph $\s{G}_\ell'(\Gamma_0(M);F)$ cannot be too close together. This structural restriction on the supersingular isogeny graph may be thought of as a kind of second-order {\it girth} condition.

Given $\tau\in Y_0(M)(F)$ and $r\ge 0$, let $\s{N}_\ell^r(\tau)$ be the subgraph of $\s{G}_\ell'(\Gamma_0(M);F)$ obtained by taking the union of all paths of length $\le r$ originating from $\tau$. We consider $\s{N}_\ell^r(\tau)$ as being rooted at $\tau$.

\begin{lemma}\label{Lemma-Local-Structure}Let $\tau\in Y_0(M)(F)$ and $r\ge 0$. Suppose either that $\tau$ is ordinary or that $\tau$ is supersingular and $p=\mathrm{char}(F)$ satisfies $p>\max\{\frac{4}{M}\ell^{4r},\ell^{2r}\}$. Then,
\begin{itemize}\item[a.] $\s{N}_\ell^r(\tau)$ contains at most one simple cycle,
\item[b.] $\s{N}_\ell^r(\tau)$ contains at most one elliptic point,
\item[c.] If $\s{N}_\ell^r(\tau)$ contains both a cycle and an elliptic point, then the cycle is a loop based at the elliptic point, and
\item[d.] If $\tau'\in\s{N}_\ell^r(\tau)$ is non-elliptic and $\mathrm{dist}(\tau,\tau')<r$, then $\deg_{\s{N}_\ell^r(\tau)}(\tau')=\ell+1$.\end{itemize}\end{lemma}

\begin{proof}If $p\in\{2,3\}$, the inequality $p>\max\{\frac{4}{M}\ell^{4r},\ell^{2r}\}$ is only satisfied when $r=0$ in which case the conclusion is trivial. Thus, we will assume without loss that either $\tau$ is ordinary or that $\tau$ is supersingular and $\mathrm{char}(F)\ne 2,3$. Let $G$ denote the connected component of $\tau$.

(a.) If $G$ is acyclic or has only one simple cycle, then we are done, so assume that $G$ has two or more simple cycles (in which case $\tau$ is supersingular). Let $s\ge 0$ be least so that $\s{N}_\ell^s(\tau)$ contains two distinct simple cycles. Then, there exist directed cycles $x$ and $y$ based at $\tau$ of length $\le 2s$ whose homotopy classes do not commute in $\pi_1(G,\tau)$. Applying $\xi_\tau$ to $x$ and $y$ yields elements $\alpha,\beta\in\End(\tau)$ such that $\alpha\beta\ne\beta\alpha$ and $\N(\alpha),\N(\beta)\le\ell^{2s}$. By Lemma \ref{Lemma-GL}, we have $Mp\le 4\ell^{4s}$. Thus, if $p>\frac{4}{M}\ell^{4r}$ we must have $r<s$, so $\s{N}_\ell^r(\tau)$ contains at most one simple cycle.

(b.)\ Suppose that $\tau_1,\tau_2\in G$ are distinct elliptic vertices. Because $\mathrm{char}(F)\ne2,3$, the group $\End(\tau_i)^\times$ is cyclic of order $4$ or $6$. Set $w(\tau_i)=\tfrac{1}{2}|\End(\tau_i)^\times|$ ($i=1,2$) and let $d=\dist_G(\tau_1,\tau_2)$. We have two cases:
\begin{itemize}\item If $w(\tau_1)\ne w(\tau_2)$ then $\Z[u_1]\hookrightarrow\End(\tau_1)$ and $\Z[u_2]\hookrightarrow\End(\tau_2)$ where $u_1$ and $u_2$ are roots of unity generating distinct quadratic extensions of $\Q$. By Lemma \ref{Lemma-End-Transfer}, there is an embedding $\Z[\ell^d u_1]\hookrightarrow\End(\tau_2)$. Since $u_1$ and $u_2$ cannot commute in the quaternion algebra $\End(\tau_2)\otimes\Q$, $\ell^d u_1$ and $u_2$ do not commute in $\End(\tau_2)$, so $\tau$ is supersingular. It follows from the Goren--Lauter lemma that $4\mathrm{N}(\ell^d u_1)\mathrm{N}(u_2)=4\ell^{2d}\ge Mp$, and therefore $d>2r$ (because $Mp>\ell^{4r}$). This proves that at most one of $\tau_1,\tau_2$ is a vertex of $\s{N}_\ell^r(\tau)$.
\item If $w(\tau_1)=w(\tau_2)$, then there exists an elliptic curve $E/F$ with $j(E)\in\{0,1728\}$ and distinct subgroups $C_1,C_2\subseteq E(F)$ cyclic of order $M$ such that $(E,C_i)$ represents $\tau_i$ ($i=1,2$). Following a path $\tau_1\rar\tau_2$ of minimal length in $G$ yields an endomorphism $\phi$ of the elliptic curve $E$ of norm $\ell^d$ such that $\phi(C_1)=C_2$. Let $u$ generate the group $\End(E)^\times=\End(E,C_1)^\times$. Since $u(C_1)\subseteq C_1$ but $\phi(C_1)\not\subseteq C_1$, we have $\phi\notin\Z[u]$. Because $\Z[u]$ is a maximal quadratic order, it follows that $\phi$ and $u$ do not commute in $\End(E)$. Thus, $\tau$ is supersingular and the Goren--Lauter lemma yields $4\mathrm{N}(\phi)\mathrm{N}(u)=4\ell^{d}\ge p$. Thus $d>2r$ (because $p>4\ell^{2r}$), and again we conclude that at most one of $\tau_1,\tau_2$ is a vertex of $\s{N}_\ell^r(\tau)$.\end{itemize}

(c.)\ Let $x$ and $\tau'$ denote be the unique simple cycle and the unique elliptic point on $\s{N}_\ell^r(\tau)$, respectively. If $a:\tau\rar\tau'$ is a path of minimal length from $\tau'$ to a vertex $\tau''$ on $x$, then $y=ax\widehat{a}$ is a directed cycle in $\s{N}_\ell^r(\tau)$  based at $\tau'$ (with either direction assigned to $x$). Let $\alpha=\xi_{\tau'}(ax\widehat{a})$, so $\mathrm{N}(\alpha)=\ell^{|x|+2d}$ where $|x|\le 2r$ and $d\le 2r$.

Fix a generator $u$ for $\End(\tau')^\times$. If $\alpha u\ne u\alpha$, then by Lemma \ref{Lemma-GL},
\[4\mathrm{N}(\alpha)\mathrm{N}(u)=4\ell^{|x|+2d}\ge Mp>4\ell^{4r}\]
which is impossible because $|x|+2d\le 4r$.

Therefore, we must have $\alpha u=u\alpha$, so $\alpha\in\Z[u]$ (since $\Z[u]$ is a maximal quadratic order). Because $ax\widehat{a}$ is not contractible, $\alpha\notin\Z$, so it follows that $\ell$ is either split or ramified in $\Z[u]$. Since $\Z[u]$ has class number $1$, it follows that there is $\pi\in\Z[u]\subseteq\End_\ell(\tau')$ such that $\mathrm{N}(\pi)=\ell$, and the coarse isomorphism class of $\pi$ (as an $\ell$-isogeny) is a loop $[\pi]:\tau'\rar\tau'$.

%It follows that in addition $d=0$ and that $x=[\pi]$.

Claim (d.) follows from the fact that all non-elliptic vertices in $\s{G}_\ell'(\Gamma_0(M);F)$ have degree $\ell+1$.\end{proof}

\subsection{Lower bounds on polar conditions}\label{S2-polar}

The next step is to use Lemma \ref{Lemma-Local-Structure} to formulate lower bounds on polar conditions $\s{G}_\ell'(\Gamma_0(M);F)$.

For $n\ge 2$ and $i>-n$, let $\s{T}(n,i)$ denote the infinite rooted tree such that the root has degree $n+i$ and every other vertex has degree $n$. If $r\ge  0$, let $\s{T}^r(n,i)$ be the subgraph of $\s{T}(n,i)$ induced on the vertex set $\set{v}{\mathrm{dist}_{\s{T}(n,i)}(v_0,v)\le r}$ where $v_0$ is the root of $\s{T}(n,i)$. The graph $\s{T}^r(n,i)$ is a full rooted tree of depth $r$.

When $\s{T}$ is a tree, we say that a nonempty vertex subset $Q\subseteq\s{T}$ is {\it quasipolar} if for all vertices $v\in\s{T}$ the condition $\deg(v,Q)=1$ implies that either $v\in Q$ or that $v$ is a {\it leaf} of $\s{T}$. If $P$ is a polar condition on $\s{T}(n,i)$, then $P\cap\s{T}^r(n,i)$ is quasipolar on $\s{T}^r(n,i)$; conversely, if $Q\subseteq\s{T}^r(n,i)$ is quasipolar, there exists a polar condition $P\subseteq\s{T}(n,i)$ such that $Q=P\cap\s{T}^r(n,i)$. We define
\[b^r(n,i)=\min_{\substack{P\subseteq\s{T}(n,i)\\\text{$P$ polar}\\v_0\in P}}|P\cap\s{T}^r(n,i)|=\min_{\substack{Q\subseteq\s{T}^r(n,i)\\\text{$Q$ quasipolar}\\v_0\in P}}|Q|\]
where $P$ ranges over all polar conditions on $\s{T}(n,i)$ that contain the root and $Q$ ranges over all quasipolar subsets of $\s{T}^r(n,i)$ that contain the root.

\begin{lemma}\label{Lemma-Polar-Trees}If $r\ge 0$, $n\ge 2$, and $i>-n$, we have
\[b^r(n,i)=1+(n+i)\sum_{j=0}^{\floor{r/2}-1}(n-1)^j\pd\]\end{lemma}
\begin{proof}First, we claim that the right hand side is a lower bound on $b^r(n,i)$.  This is trivial when $r=0$ or $r=1$. Suppose that $r\ge 2$ and that $P$ is a polar condition on $\s{T}(n,i)$ that contains the root. If $v$ is any vertex of $\s{T}^r(n,i)$ let $\s{T}[v]$ denote the subtree of $\s{T}^r(n,i)$ rooted at $v$ and containing all descendants of $v$. Let $v_1,\ldots,v_{n+i}$ denote the daughters of the root $v_0$ and observe that for each $k$, either $v_k\in P$ or $v_k'\in P$ for some daughter $v_k'$ of $v_k$. In the former case, $\s{T}[v_k]\simeq\s{T}^{r-1}(n,-1)$ and $|P\cap\s{T}[v_k]|\ge b^{r-1}(n,-1)$; in the latter case, $\s{T}[v_k']\simeq\s{T}^{r-2}(n,-1)$ and $|P\cap\s{T}[v_k]|\ge |P\cap\s{T}[v_k']|\ge b^{r-2}(n,-1)$. Since $b^{r-1}(n,-1)\ge b^{r-2}(n,-1)$ and $v_0\in P$, it follows that
\[|P\cap\s{T}^r(n,i)|=1+\sum_{k=1}^{n+i}|P\cap\s{T}[v_k]|\ge 1+(n+i)b^{r-2}(n,-1)\text{,}\]
and our claim follows by induction.

On the other hand, a straightforward construction yields a polar condition $P$ on $\s{T}(n,i)$ with
\[|P\cap\set{v}{\dist_{\s{T}(n,i)}(v_0,v)=r}|=\left\{\begin{array}{ll}1&\text{if $r=0$,}\\
(n+i)(n-1)^{r/2-1}&\text{if $r$ is even and $r\ge 2$,}\\
0&\text{if $r$ is odd,}\end{array}\right.\]
from which we conclude that the right hand side is also an upper bound on $b^r(n,i)$.\end{proof}

A graph containing a unique simple cycle is called a {\it volcano} (see for example \cite{BLS}); the {\it crater} of a volcano is its unique simple cycle. For $n\ge 3$ and $c\ge 1$ let $\s{V}(n,c)$ denote the infinite $n$-regular volcano with a crater of length $c$, rooted at some vertex $v_0$ on the crater (the particular choice of root being otherwise unimportant). For $r\ge 0$, let $\s{V}^r(n,c)$ be the subgraph of $\s{V}(n,c)$ induced on the vertex set $\set{v}{\dist_{\s{V}(n,c)}(v_0,v)\le r}$.

%As a general principle, a vertex subset $P'\subseteq G$ extends to a polar condition 

\begin{lemma}\label{Lemma-Polar-Volcanoes}Let $r\ge 0$, $n\ge 3$, and $c\ge 3$. Suppose that $P$ is a polar condition on $\s{V}(n,c)$ containing the root $v_0$. If $r\ge 2$, then
\[|P\cap\s{V}^r(n,c)|\ge (n-1)^{\floor{r/2}}+(n-1)^{\floor{r/2}-1}\]
\end{lemma}

\begin{proof}Let $x$ be the crater of $\s{V}(n,c)$.

\begin{figure}[h]\begin{center}\includegraphics[scale=0.5]{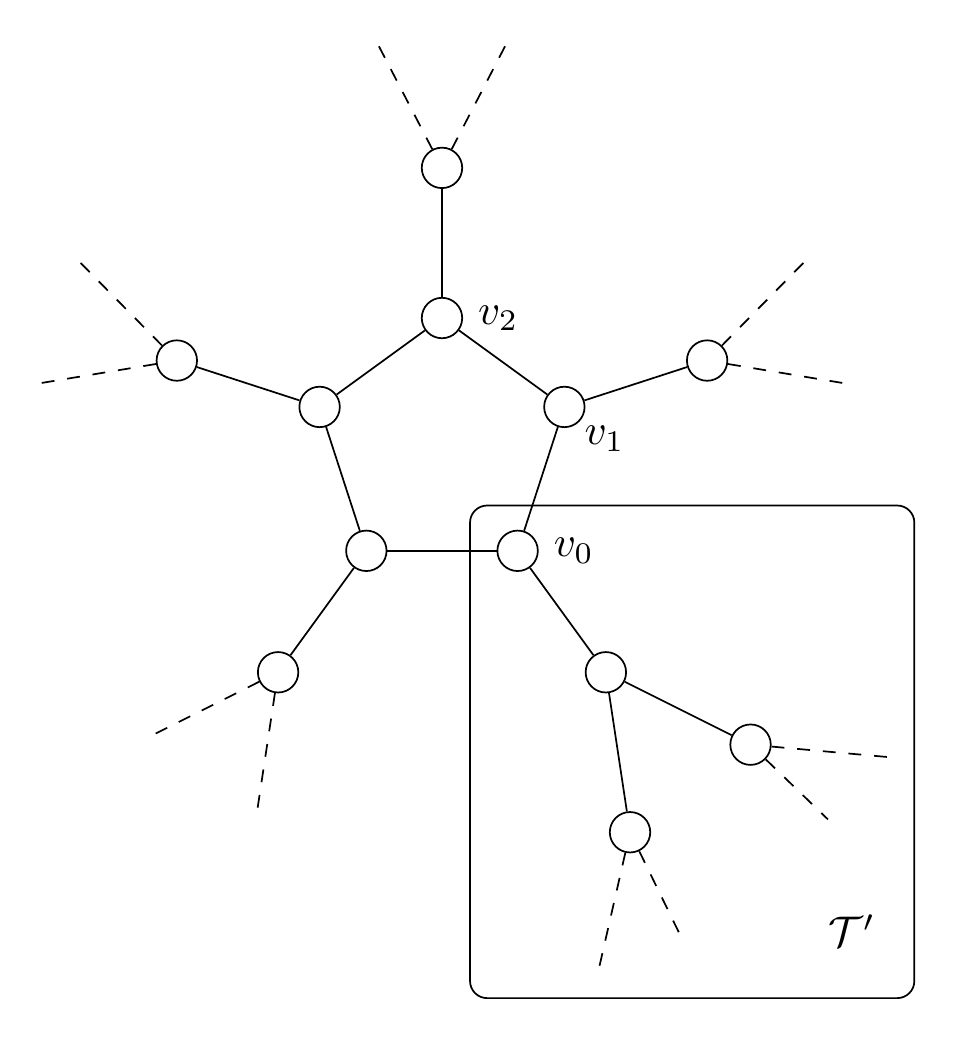}\end{center}\end{figure}

For any vertex $v$ of $\s{V}^r(n,c)$ let $\s{T}[v]$ denote the subtree of $\s{V}^r(n,c)$ rooted at $v$ and containing all descendants of $v$ (in the figure $\s{T}'=\s{T}[v_0]$); formally, $\s{T}[v]$ is the subgraph induced on the vertex set
\[\set{w}{\text{$\dist(x,w)\ge\dist(x,v)$ and every path $w\rar x$ contains $v$}}\]
where $x$ is the crater of $\s{V}(n,c)$.

Let $\s{T}'=\s{T}[v_0]$ (as in the figure) and note that $\s{T}'\simeq\s{T}^r(n,-2)$.

\noindent Choose vertices $v_1$ and $v_2$ on the crater so that there is a path $v_0\rar v_1\rar v_2$ of length $2$ (as in the figure). Remembering that $v_0\in P$, we define $\s{T}''$ according to three cases:
\begin{itemize}\item[i.] If $v_1\in P$, let $\s{T}''=\s{T}[v_1]$. In this case $\s{T}''\simeq\s{T}^{r-1}(n,-2)$.
\item[ii.] If $v_1\notin P$, and there exists a daughter $v_1'$ of $v_1$ on $\s{T}[v_1]$ such that $v_1'\in P$, let $\s{T}''=\s{T}[v_1']$, so $\s{T}''\simeq\s{T}^{r-2}(n,-1)$.
\item[iii.] Otherwise we must have $v_2\in P$ since $P$ is a polar condition and $v_0\in P$. In this case, let $\s{T}''=\s{T}[v_2]$, so $\s{T}''\simeq\s{T}^{r-2}(n,-2)$.\end{itemize}

In each of these cases, $P\cap\s{T}'$ and $P\cap\s{T}''$ are quasipolar and contain the roots of $\s{T}'$ and $\s{T}''$, respectively. Thus,
\[|P\cap\s{V}^r(n,c)|\ge|P\cap\s{T}'|+|P\cap\s{T}''|\ge b^r(n,-2)+\left\{\begin{array}{ll}b^{r-1}(n,-2)&\text{in case (i.),}\\
b^{r-2}(n,-1)&\text{in case (ii.),}\\
b^{r-2}(n,-2)&\text{in case (iii.).}\end{array}\right.\]
The bound in case (iii.) is the weakest, so Lemma \ref{Lemma-Polar-Trees} yields
\[|P\cap\s{V}^r(n,c)|\ge b^r(n,-2)+b^{r-2}(n,-2)=(n-1)^{\floor{r/2}}+(n-1)^{\floor{r/2}-1}\pd\]\end{proof}

%
%
%$v_1\in P$; in this case, $\s{T}[v_1]\simeq\s{T}^{r-1}(n,-2)$, and
%\[|P\cap\s{V}^r(n,c)|\ge |P\cap\s{T}[v_0]|+|P\cap\s{T}[v_1]|\ge b^r(n,-2)+b^{r-1}(n,-2)\]
%\item There is a daughter $v_1'$ of $v_1$ on $\s{T}[v_1]$ such that $v_1'\in P$; here, $\s{T}[v_1']\simeq\s{T}^{r-2}(n,-1)$, and
%\[|P\cap\s{V}^r(n,c)|\ge |P\cap\s{T}[v_0]|+|P\cap\s{T}[v_1']|\ge b^r(n,-2)+b^{r-2}(n,-1)\]
%\item $v_2\in P$; in this case, $\s{T}[v_1]\simeq\s{T}^{r-1}(n,-2)$, and
%\[|P\cap\s{V}^r(n,c)|\ge |P\cap\s{T}[v_0]|+|P\cap\s{T}[v_2]|\ge b^r(n,-2)+b^{r-2}(n,-2)\]
%\end{itemize}
%Of these three bounds, the last one is the weakest, so
%\[|P\cap\s{V}^r(n,c)|\ge b^r(d,-2)+b^{r-2}(n,-2)=(n-1)^{\floor{r/2}}+(n-1)^{\floor{r/2}-1}\text{.}\]\end{proof}

\begin{lemma}\label{Lemma-Admissible}Let $G$ be a connected component of $\s{G}'_\ell(\Gamma_0(M);F)$ and let $P$ be a polar condition on $G$. Suppose either that $G$ is ordinary or that $G$ is supersingular and $p=\mathrm{char}(F)$ satisfies $p>\max\{\frac{4}{M}\ell^8,4\ell^4\}$.

There exists $\tau\in P$ such that $\tau$ is not elliptic and $\tau$ lies on no cycles of length $\le 2$.\end{lemma}

\begin{proof}Fix $\tau_0\in P$ and assume without loss that $\tau_0$ is either elliptic or that it lies on a cycle of length $\le 2$. Applying Lemma \ref{Lemma-Local-Structure} with $r=2$ and some simple arguments using the fact that $P$ is a polar condition (like those in the proof of Lemma \ref{Lemma-Polar-Volcanoes}) yield $\tau\in P$ with the desired properties.\end{proof}

\begin{lemma}\label{Lemma-Combinatorial}Let $G$ be a connected component of $\s{G}_\ell(\Gamma_0(M);F)$ and let $P$ be a polar condition on $G$.
\begin{itemize}\item[a.]If $G$ is ordinary then $P$ is infinite.
\item[b.]If $G$ is supersingular and $p=\mathrm{char}(F)$, then for all $r\ge 2$ satisfying $p>\max\{\frac{4}{M}\ell^{4r},4\ell^{2r}\}$ we have $|P-{W}_M|\ge\ell^{\floor{r/2}}+\ell^{\floor{r/2}-1}$ where ${W}_M$ is the set of elliptic points on $Y_0(M)(F)$.\end{itemize}\end{lemma}

\begin{proof}By Lemma \ref{Lemma-Isogeny-Graph-Transfer} it is sufficient to prove the result upon replacing $\s{G}_\ell(\Gamma_0(M);F)$ with the undirected graph $\s{G}_\ell'(\Gamma_0(M);F)$. Let $G$ be a connected component of $\s{G}_\ell'(\Gamma_0(M);F)$, let $r\ge 2$, and assume either that $G$ is ordinary or that $G$ is supersingular and $p>\max\{\frac{4}{M}\ell^{4r},4\ell^{2r}\}$.

By Lemma \ref{Lemma-Admissible} we may choose $\tau\in P$ such that $w(\tau)=1$ and such that $\tau$ does not lie on any cycles of length $\le 2$. Let $G_0=\s{N}_\ell^r(\tau)$. Following Lemma \ref{Lemma-Local-Structure} there are four cases:

\begin{itemize}\item ($G_0$ contains no cycles and no elliptic points.) In this case, $G_0$ is isomorphic to $\s{T}^r(\ell+1,0)$ as a rooted graph and $P\cap G_0$ is quasipolar on $G_0$, so $|P\cap G_0|\ge b^r(\ell+1,0)$. Applying \ref{Lemma-Polar-Volcanoes},
\[|P\cap G_0|\ge\ell^{\floor{r/2}}+\ell^{\floor{r/2}-1}\]
and $P\cap G_0$ contains no elliptic points.

\item ($G_0$ contains a cycle $x$ and $\tau\in x$.) In this case, $G_0$ is isomorphic to $\s{V}^r(\ell+1,|x|)$ as a rooted graph and $P\cap G_0$ extends to a polar condition on $\s{V}(\ell+1,|x|)$. Since $|x|\ge 3$, Lemma \ref{Lemma-Polar-Volcanoes} guarantees that $|P\cap G_0|\ge\ell^{\floor{r/2}}+\ell^{\floor{r/2}-1}$ and that $P\cap G_0$ contains no elliptic points.
\item ($G_0$ contains a cycle $x$ and $\tau\notin x$.) Because $G_0$ contains exactly one cycle, there is a unique shortest path $a:\tau\rar x$. Let $e$ be the first edge on this path, and consider $G_0-\{e\}$. The connected component $G_1$ of $G_0-\{e\}$ containing and rooted at $\tau$ is isomorphic as a rooted graph to $\s{T}^r(\ell+1,-1)$. $P\cap G_1$ is quasipolar on $G_1$, so by Lemma \ref{Lemma-Polar-Trees},
\[|P\cap G_0|\ge |P\cap G_1|\ge b^r(\ell+1,-1)\ge \ell^{\floor{r/2}}+\ell^{\floor{r/2}-1}\]
and $P\cap G_1$ contains no elliptic points (by Lemma \ref{Lemma-Local-Structure}).
\item ($G_0$ contains an elliptic point $\tau'$.) Since $\tau\ne\tau'$, we may proceed as in the previous case with $\tau'$ replacing $x$.\end{itemize}
We conclude in every case that $|P\cap G-{W}_M|\ge\ell^{\floor{r/2}}+\ell^{\floor{r/2}-1}$. This proves (b.) directly and it proves (a.) by taking $r\rar\infty$.\end{proof}

%\subsection{Proofs of Theorems \ref{HSThm1} and \ref{HSThm2}}\label{S2-proofs}

\begin{proof}[Proof of Theorem \ref{HSThm1}]Suppose that $V\subseteq\M^*(N;F)$ is stable under the action of the Hecke operator $T_\ell$. The vertex set $\delta_{N,N}\Pi(V)$ is a polar condition on $\s{G}_\ell(\Gamma_0(N);F)$ by Proposition \ref{Proposition-Pi-Polar}.

If there is a cusp $\tau\in X_0(N)(F)$ and $f\in V$ such that $f(\tau)=\infty$, then $V$ is infinite-dimensional by Lemma \ref{Lemma-No-Cusps}. If there is an ordinary $\tau\in Y_0(N)(F)$ and $f\in V$ such that $f(\tau)=\infty$, then an ordinary component $G$ of $\s{G}_\ell(\Gamma_0(N);F)$ meets the polar condition $\delta_{N,N}\Pi(V)$. Since $G\cap \delta_{N,N}\Pi(V)$ is a polar condition on the ordinary component $G$, it is infinite (Lemma \ref{Lemma-Combinatorial}.a). It follows that $\Pi(V)$ is infinite, so $V$ is infinite-dimensional. This proves Theorem \ref{HSThm1}.a.

By the preceding argument, if $V$ is finite-dimensional and $T_\ell$-stable, $\Pi(V)$ consists of supersingular points. Since the Hasse invariant $A$ has a simple root at every supersingular point on $X_1(N)$, there is $r\ge 0$ large enough so that $A^r\cdot V$ contains no modular ratios with poles. This proves Theorem \ref{HSThm1}.b.\end{proof}

\begin{proof}[Proof of Theorem \ref{HSThm2}]With notation as in the statement of the theorem, we have
\[\Pi(V_{\Lambda,\ell}'(F))\subseteq\mathrm{Z}(\Lambda)\pd\]
The sufficiency of conditions (i.) and (ii.) follow directly from Theorem \ref{HSThm1}.a. The sufficiency of condition (iii.) follows from Lemma \ref{Lemma-Combinatorial}.b.\end{proof}

\section{Hecke stability and computation}\label{S3}

We will now demonstrate how to use the Hecke stability theorems to compute spaces of weight $1$ modular forms.

Using the $q$-expansion map at our chosen cusp we will identify modular ratios over a field $\kappa$ (not necessarily algebraically closed) with their images in $\kappa\dblparens{q}$ under $q$-expansion. For a given $P\in\Z$ we let $(q^P)$ denote the subspace of $\kappa\dblparens{q}$ spanned by $\{q^n\}_{n\ge P}$. To {\it compute} a finite-dimensional $W\subseteq F\dblparens{q}$ is to give an algorithm that on input $P\in\Z_{\ge 0}$ outputs a basis for $W\modulo(q^P)$---that is, a basis for $W$ {\it computed to precision $P$}.

%We will now explain how to use the Hecke stability theorems (Theorems x and x) to compute spaces of the form $\S_1(N,\chi;F)$ where $F=\Qbar$ and $F=\Fbar{p}$ where $p\nmid N$ and $\chi:(\Z/N\Z)^\times\rar F^\times$ is an odd character.
%
%For simplicity, fix a cusp $\tau\in X_1(N)(F)$ so that the image of $q$-expansion $\M_1^*(N;F)\rar F((q^{1/N}))$ at that cusp lies in $F\dblparens{q}$. Because we are working with modular ratios of a particular weight, we may unambiguously identify subspaces of $\M_1^*(N;F)$ with their images under $q$-expansion at $\tau$.

%For any $\theta:(\Z/N\Z)^\times\rar F^\times$, let $T_\ell^\theta$ denote the operator on $F\dblparens{q}$ given by
%\[\sum_{n\in\Z}a_nq^n\mapsto \sum_{n\in\Z}a_{\ell n}q^n+\theta(\ell)\sum_{n\in\Z}a_nq^{\ell n}\pd\]
%$T_\ell^\theta$ coincides with the Hecke operator $T_\ell$ on the subspace $\M_1^*(N,\theta;F)$ of $F\dblparens{q}$.

Fix a choice of level $N\ge 1$, a character $\chi:(\Z/N\Z)^\times\rar\Qbar^\times$, and a prime $\ell\nmid N$. Let $K=\Q(\chi)$, and for every nonzero prime ideal $\f{p}\subseteq\s{O}_K$ (including zero) let $\kappa_\f{p}=\s{O}_\f{p}/\f{p}$ (here $\s{O}_\f{p}$ just denotes the $\f{p}$-integral subring of $K$; $\s{O}_0=\kappa_0=K$). The goal of the Hecke stability method is to compute for (almost) all prime ideals $\f{p}\nmid N$, a finite-dimensional $T_\ell$-stable space $V'(\kappa_\f{p})$ of modular ratios such that
\[\S_1(N,\chi;\kappa_\f{p})\subseteq V'(\kappa_\f{p})\subseteq\M_1^*(N,\chi;\kappa_\f{p})\]
where $\chi$ is also used to denote the character obtained by composing with reduction mod $\f{p}$. Theorem \ref{HSThm2} guarantees that $V'(\kappa_0)$ consists of modular forms, but the analogous statements for nonzero $\f{p}$ must be certified.

%For any $\theta:(\Z/N\Z)^\times\rar F^\times$, let $T_\ell^\theta$ denote the operator on $F\dblparens{q}$ given by
%\[\sum_{n\in\Z}a_nq^n\mapsto \sum_{n\in\Z}a_{\ell n}q^n+\theta(\ell)\sum_{n\in\Z}a_nq^{\ell n}\pd\]
%$T_\ell^\theta$ coincides with the Hecke operator $T_\ell$ on the subspace $\M_1^*(N,\theta;F)$ of $F\dblparens{q}$.

\subsection{Integral subspace operations}\label{S3-subspace-ops}

To simplify the exposition of the next few sections, we introduce the notion of an integral subspace operation on Laurent series. Let $W$ be a finite-dimensional subspace of $K\dblparens{q}$ and let $\f{p}$ be a prime of $K$ (possibly zero). Define $\Red_\f{p}(W)$ to be the subspace of $\kappa_\f{p}\dblparens{q}$ obtained by reducing $W\cap \s{O}_\f{p}\dblparens{q}$ modulo $\f{p}$. We say that $W$ has {\it good reduction at $\f{p}$} if $\dim_K(W)=\dim_{\kappa_\f{p}}\Red_\f{p}(W)$.

\begin{definition}\label{Definition-Subspace-Operation}Let $S$ be a set of prime ideals of $K$ such that $0\in S$ (in applications, $S$ will consist of ``good'' primes). An {\it $S$-integral subspace operation} in $K\dblparens{q}$ is a family $\s{F}=\{\s{F}_\f{p}\}_{\f{p}\in S}$ of maps where for each $\f{p}$,
\[\s{F}_\f{p}:\{\text{finite-dimensional subspaces of $\kappa_\f{p}\dblparens{q}$}\}\rar\{\text{finite-dimensional subspaces of $\kappa_\f{p}\dblparens{q}$}\}\]
Satisfying the following: For all $\f{p}$,
\begin{itemize}\item The map $\s{F}_\f{p}$ is monotonic with respect to containment, and
\item If $V\subseteq K\dblparens{q}$ is finite-dimensional and $V$ has good reduction at $\f{p}$, then $\Red_\f{p}\s{F}_0(V)\subseteq\s{F}_\f{p}\Red_\f{p}(V)$.\end{itemize}
\end{definition}

We are primarily concerned with two kinds of subspace operations:
\begin{itemize}\item {\it Intersection operations.} Let $U$ be a subspace of $K\dblparens{q}$ that has good reduction at every nonzero prime ideal in $S$. Define the {\it intersection operation} (with $U$) $\s{I}^U=\{\s{I}_\f{p}^U\}_{\f{p}\in S}$ by $\s{I}_{\f{p}}^U(W)=W\cap\Red_\f{p}(U)$ for all $\f{p}\in S$ and any finite-dimensional $W\subseteq\kappa_\f{p}\dblparens{q}$.
\item {\it Stabilization operations.} Let $T$ be a linear operator on $K\dblparens{q}$ that restricts to an $\s{O}_\f{p}$-module homomorphism $\s{O}_\f{p}\dblparens{q}\rar\s{O}_\f{p}\dblparens{q}$ for all nonzero $\f{p}\in S$. Then, for all $\f{p}\in S$, there exists a unique linear transformation $T:\kappa_\f{p}\dblparens{q}\rar\kappa_\f{p}\dblparens{q}$ such that $\Red_\f{p}\circ T=T\circ\Red_\f{p}$. Define the {\it stabilization operation} (with respect to $T$) $\s{S}^T=\{\s{S}^T_\f{p}\}_{\f{p}\in S}$ by $\s{S}^T_\f{p}(W)=\set{g\in W}{Tg\in W}$ for all $\f{p}\in S$ and any finite-dimensional $W\subseteq \kappa_\f{p}\dblparens{q}$.
\end{itemize}

It is a straightforward exercise to verify that intersection operations and stabilization operations are subspace operations. The uniqueness of $T:\kappa_\f{p}\dblparens{q}\rar\kappa_\f{p}\dblparens{q}$ above is not entirely trivial, but it follows from the finiteness of $\mathrm{Cl}(K)$.

\subsection{The Hecke stability method via integral subspace operations}\label{S3-HSM}

In this section we present the theoretical details of the computation described in Section \ref{S1-wt1}. This computation can be expressed as a composition of intersection operations and stabilization operations on Laurent series, introduced in the previous section. The practical details (precision requirements and linear algebra) of the computations are left to Sections \ref{S3-precision} and \ref{S3-subroutines} (respectively).

We begin by fixing a finite nonempty $\Lambda\subseteq\M_1^*(N,\chi^{-1};K)-\{0\}\subseteq K\dblparens{q}$. It is ideal (but not necessary) that $\mathrm{Z}(\Lambda)\subseteq X_1(N)(K)$ be as small as possible and contain no cusps (so that elements of $V'_\Lambda(K)$ as defined in Section \ref{S1-wt1} vanish at the cusps). Index $\Lambda=\{\lambda_0,\ldots,\lambda_s\}$ and let $U_i=\im[\lambda_i^{-1}]$ where for each $\lambda\in\Lambda$,
\[[\lambda^{-1}]:\S_2(N,\boldsymbol{1};K)\rar\M_1^*(N,\chi;K):g\mapsto g/\lambda\pd\]

For all primes $\f{p}\nmid N$, the reduction map $\S_2(N,\boldsymbol{1};\s{O}_K[\frac{1}{N}])\rar\S_2(N,\boldsymbol{1};\kappa_\f{p})$ is surjective, so $\S_2(N,\boldsymbol{1};K)$ has good reduction everywhere. Let
\[\f{b}=\prod_{\substack{\f{p}\nmid N\\\exists\lambda\in\Lambda\ \lambda\equiv 0\modulo\f{p}}}\f{p}\pd\]
Note that if $\f{p}\nmid N\f{b}$, every member of $\{U_i\}_{0\le i\le s}$ has good reduction at $\f{p}$.

Choose $\ell\nmid N$ and let $S=\set{\f{p}\subseteq\s{O}_K}{\f{p}\nmid \ell N\f{b}}$. Denote by $T_\ell^\chi$ the operator on $K\dblparens{q}$ given by
\[\sum_{n\in\Z}a_nq^n\mapsto\sum_{n\in\Z}a_{\ell n}q^n+\chi(\ell)\sum_{n\in\Z}a_nq^{\ell n}\pd\]$T_\ell^\chi$ coincides with $T_\ell$ on $\M_1^*(N,\chi;F)$ (see Section \ref{S2-ratios}) and it restricts to a module homomorphism $\s{O}_\f{p}\dblparens{q}\rar\s{O}_\f{p}\dblparens{q}$ for any $\f{p}\in S$. We set the following notation:

\begin{itemize}\item For all $i\ge 1$, let
\[\s{F}^{(i)}=\left\{\begin{array}{ll}\s{I}^{U_i}&\text{if $i\le s$}\\
\s{S}^{T_\ell^\chi}&\text{otherwise,}\end{array}\right.\]

\item Let $V_{\Lambda,\ell}^{(0)}(\kappa_\f{p})=\Red_\f{p}(U_0)$, and for all $i\ge 1$, let $V_{\Lambda,\ell}^{(i)}(\kappa_\f{p})=\s{F}_\f{p}^{(i)}V_{\Lambda,\ell}^{(i-1)}(\kappa_\f{p})$.\end{itemize}

Because $\dim V_{\Lambda,\ell}^{(0)}(\kappa_\f{p})=\dim\S_2(N,\boldsymbol{1};\kappa_\f{p})$ is finite and independent of $\f{p}$, there exists $t$ such that $V_{\Lambda,\ell}^{(s+t)}(\kappa_\f{p})$ is a $T_\ell$-stable subspace of $\M_1^*(N,\chi;\kappa_\f{p})$ for all $\f{p}$. In the notation of Section \ref{S1-wt1} (which we will continue to use), $V_\Lambda'(\kappa_\f{p})=V_{\Lambda,\ell}^{(s)}(\kappa_\f{p})$ and $V_{\Lambda,\ell}'(\kappa_\f{p})=V_{\Lambda,\ell}^{(s+t)}(\kappa_\f{p})$. For each $\f{p}$ the space $V_{\Lambda,\ell}'(\kappa_\f{p})$ is $T_\ell$-stable and contains $\S_1(N,\chi;\kappa_{p})$. Schematically,

\[\xymatrixcolsep{3pc}\xymatrix{U_0\ar@{|->}^{\Red_\f{p}}[d]\\\Red_\f{p}(U_0)\ar@{=}[d]\ar@{|-->}[rrr]^{\text{intersection ops.}}&&&V_\Lambda'(\kappa_\f{p})\ar@{=}[d]\ar@{|-->}[rr]^{\text{stabilization ops.}}&&V_{\Lambda,\ell}'(\kappa_\f{p})\ar@{=}[d]\\V_{\Lambda,\ell}^{(0)}(\kappa_\f{p})\ar@{|->}[r]^{\s{I}_\f{p}^{U_1}}_{\s{F}_\f{p}^{(1)}}&V_{\Lambda,\ell}^{(1)}(\kappa_\f{p})\ar@{|->}[r]&\ \cdots\ \ar@{|->}[r]_{\s{F}_\f{p}^{(s)}}^{\s{I}_\f{p}^{U_s}}&V_{\Lambda,\ell}^{(s)}(\kappa_\f{p})\ar@{|->}[r]_{\s{F}_\f{p}^{(s+1)}}^{\s{S}_\f{p}^{T_\ell^\chi}}&\ \cdots\ \ar@{|->}[r]_{\s{F}_\f{p}^{(s+t)}}^{\s{S}_\f{p}^{T_\ell^\chi}}&V_{\Lambda,\ell}^{(s+t)}(\kappa_\f{p})}\]
for each $\f{p}\in S$.
%where the initial space $\Red_\f{p}(U_0)$ is obtained by dividing $\S_2(N,\boldsymbol{1};\Z)$ by $\lambda_0$ and reducing mod $\f{p}$.

\subsection{Nonsurjectivity of reduction}\label{S3-detection}

Next, we will show how the Hecke stability method can be used to produce a family  $\{V'(\kappa_\f{p})\}_{\f{p}\in S}$ where each $V'(\kappa_\f{p})$ is a $T_\ell$-stable subspace of $\M_1^*(N,\chi;\kappa_\f{p})$ containing $\S_1(N,\chi;\kappa_\f{p})$.

First, recall that for {almost all} nonzero $\f{p}\in S$ the reduction map $\S_1(N,\chi;\s{O}_\f{p})\rar\S_1(N,\chi;\kappa_\f{p})$ is surjective. For such $\f{p}$, taking $V'(\kappa_\f{p})=\Red_\f{p}V_{\Lambda,\ell}'(K)$ works, provided that $V_{\Lambda,\ell}'(K)$ also has good reduction at $\f{p}$. We therefore only need to compute $V_{\Lambda,\ell}'(\kappa_\f{p})$ {\it directly} for $\f{p}=0$, for the (finitely many) $\f{p}$ at which reduction is nonsurjective, and for the (finitely many) $\f{p}$ at which $V_{\Lambda,\ell}'(\kappa_0)$ has bad reduction.

Determining the list of primes at which reduction is nonsurjective provides the most difficulty. The idea is that when surjectivity of $\S_1(N,\chi;\s{O}_\f{p})\rar\S_1(N,\chi;\kappa_\f{p})$ fails, the surjectivity of
\[V_{\Lambda,\ell}^{(i)}(K)\cap\s{O}_\f{p}\dblparens{q}\rar V_{\Lambda,\ell}^{(i)}(\kappa_\f{p})\]
must fail for some index $i$.

\begin{definition}Let $\s{F}$ be an $S$-integral subspace operation, let $V$ be a finite-dimensional subspace of $K\dblparens{q}$, let $\f{p}\in S$, and suppose that both $V$ and $\s{F}_0(V)$ have good reduction at $\f{p}$.

We say that $\f{p}$ {\it divides} $\s{F}$ at $V$ if $\Red_\f{p}\s{F}_0(V)\subsetneq\s{F}_\f{p}\Red_\f{p}(V)$, i.e., when the inclusion from Definition \ref{Definition-Subspace-Operation} is proper.\end{definition}

Recall that $\f{p}$ is called an {\it Eisenstein congruence prime} for $(k,N,\chi)$ if there exists an Eisenstein series and a cusp form of that type that are congruent to each other modulo $\f{p}$. Such $\f{p}$ divide the numerator of $\frac{1}{k}\mathrm{B}_{k,\chi}=-\mathrm{L}(1-k,\chi)$.

\begin{proposition}Fix $(N,\chi,\Lambda,\ell)$ and $S$ as above.

Let $L=L'\cup L''\cup L'''$ where
\begin{gather*}L'=\set{\f{p}\in S}{\text{there is $j$ such that $\f{p}$ divides $\s{F}^{(j+1)}$ at $V_{\Lambda,\ell}^{(j)}(K)$}}\cm\\
L''=\set{\f{p}\in S}{\text{there is $j$ such that $V_{\Lambda,\ell}^{(j)}(K)$ has bad reduction at $\f{p}$}}\text{, and}\\
L'''=\set{\f{p}\in S}{\text{$\f{p}$ is an Eisenstein congruence prime for $(1,N,\chi)$}}\pd
\end{gather*}
Then $L$ is finite and it contains all $\f{p}\in S$ such that $\S_1(N,\chi;\s{O}_{\f{p}})\rar\S_1(N,{\chi};\kappa_\f{p})$ is not surjective.
\end{proposition}

\begin{proof}Suppose that $\S_1(N,\chi;\s{O}_{\f{p}})\rar\S_1(N,{\chi};\kappa_\f{p})$ is not surjective and, without loss, that $\f{p}\notin L''\cup L'''$. Because $\f{p}\notin L'''$, there exists $f\in\S_1(N,\chi;\kappa_\f{p})$ that does not lift to any $F\in\M_1(N,\chi;\s{O}_\f{p})$. Since $V_{\Lambda,\ell}^{(s+t)}(K)$ has good reduction at $\f{p}$, Hecke stability guarantees
\begin{gather*}V_{\Lambda,\ell}^{(s+t)}(K)\cap\s{O}_\f{p}\dblparens{q}\subseteq\M_1(N,\chi;\s{O}_\f{p})\text{ by Theorem \ref{HSThm2}.a, and}\\\S_1(N,\chi;\kappa_\f{p})\subseteq V_{\Lambda,\ell}^{(s+t)}(\kappa_\f{p})\cm\end{gather*}
and it follows that $\Red_\f{p}V_{\Lambda,\ell}^{(s+t)}(K)\subsetneq V_{\Lambda,\ell}^{(s+t)}(\kappa_\f{p})$.

On the other hand, $\Red_{\f{p}}V_{\Lambda,\ell}^{(0)}(K)=V_{\Lambda,\ell}^{(0)}(\kappa_\f{p})$ and---since $\f{p}\notin L''$ and each $\s{F}^{(i)}$ is an $S$-integral subspace operation---$\Red_{\f{p}}V_{\Lambda,\ell}^{(i)}(K)\subseteq V_{\Lambda,\ell}^{(i)}(\kappa_\f{p})$ for all $i$. Hence, there exists a least $j$ satisfying
\[\Red_\f{p} V_{\Lambda,\ell}^{(j)}(K)=V_{\Lambda,\ell}^{(j)}(\kappa_\f{p})\quad\text{and}\quad\Red_\f{p} V_{\Lambda,\ell}^{(j+1)}(K)\subsetneq V_{\Lambda,\ell}^{(j+1)}(\kappa_\f{p})\pd\]
By definition, $V_{\Lambda,\ell}^{(j+1)}(K)=\s{F}^{(j+1)}_0 V_{\Lambda,\ell}^{(j)}(K)$ and $V_{\Lambda,\ell}^{(j+1)}(\kappa_\f{p})=\s{F}_\f{p}^{(j+1)}\Red_\f{p} V^{(j)}_{\Lambda,\ell}(K)$. In summary, $V_{\Lambda,\ell}^{(j)}(K)$ and $\s{F}_0^{(j+1)}V_{\Lambda,\ell}^{(j)}(K)$ both have good reduction at $\f{p}$, but the containment 
\[\Red_\f{p}\s{F}_0^{(j+1)}V_{\Lambda,\ell}^{(j)}(K)\subseteq \s{F}_\f{p}^{(j+1)}\Red_\f{p}V_{\Lambda,\ell}^{(j)}(K)\]
is proper. Schematically,
\[\xymatrix{\cdots\ \ar@{|->}[r]&V_{\Lambda,\ell}^{(j)}(K)\ar@{|->}[r]^{\s{F}_0^{(j+1)}}\ar@{|->}[d]^{\Red_\f{p}}& V_{\Lambda,\ell}^{(j+1)}(K)\ar@{|->}[r]\ar@{|->}[d]^{\Red_\f{p}}&\ \cdots\\
&\Red_\f{p}V_{\Lambda,\ell}^{(j)}(K)\ar@{=}[d]& \Red_\f{p}V_{\Lambda,\ell}^{(j+1)}(K)\ar@{_{(}-->}[d]&&\\
\cdots\ \ar@{|->}[r]&V_{\Lambda,\ell}^{(j)}(\kappa_\f{p})\ar@{|->}[r]_{\s{F}_\f{p}^{(j+1)}}& V_{\Lambda,\ell}^{(j+1)}(\kappa_\f{p})\ar@{|->}[r]&\ \cdots}\]
where the broken hooked arrow indicates proper inclusion.

Hence, $\f{p}$ divides $\s{F}^{(j+1)}$ at $V^{(j)}_{\Lambda,\ell}(K)$, so $\f{p}\in L'$.\end{proof}

\subsection{Precision requirements}\label{S3-precision}

%In this section our goal is to determine exactly how large the precision $P$ must be for our computations to be valid.

%To {\it compute} a finite-dimensional $W\subseteq\kappa\dblparens{q}$ is to give an algorithm which on input $P$ (the desired {\it precision}) produces a basis for $W\modulo(q^P)$.

The subspace operations $\{\s{F}^{(i)}\}_{i}$ in the previous section can be described easily in terms of linear algebra on Laurent series. Since $\S_2(N,\boldsymbol{1};\Z[\tfrac{1}{N}])$ and $\Lambda$ can be computed to arbitrarily high precision, for each $\f{p}$, $V_{\Lambda,\ell}'(\kappa_\f{p})$ can be computed to precision $P$ provided that $P$ is large enough. The goal of this section is to determine exactly how large $P$ must be. For simplicity, we confine ourselves to working over the field $K$; all of these results adapt easily to the mod $\f{p}$ setting.

%
%
% algorithm which computes that space to arbitrary precision.

% on input $P$ produces a generating set for the image of $W$ under {\it truncation} $\kappa\dblparens{q}\rar \kappa\dblparens{q}/(q^P)$. In this section, we determine how large $P$ must be for the computations on modular ratios described in x to be valid.

%To compute with modular forms practically, we must work with {\it truncated} Laurent series: spaces of the form $\kappa\dblparens{q}/(q^P)$. The action of the subspace operations 

Recall the {\it Sturm bound}: $\Sturm_k(N)=\frac{k}{12}[\SL_2(\Z):\Gamma_0(N)]+1$. If $P\ge\Sturm_k(N)$, then truncated $q$-expansion of modular {\it forms} $\M_k(N,\theta;\kappa)\rar \kappa\dblbrackets{q}/(q^P)$ is injective for any field $\kappa$ and any character $\theta:(\Z/N\Z)^\times\rar \kappa^\times$.

\begin{lemma}Let $i,j\in\{0,\ldots,s\}$. If $P\ge\Sturm_3(N)$ and $U_i$ and $U_j$ (as defined above) can be computed to precision $P$, then $U_i\cap U_j$ can be computed to precision $P$.\end{lemma}

%\begin{proof}It is clear that we can compute the intersection of $U_i\modulo q^P$ and $U_j\modulo q^P$, but we must verify that this is equal to $(U_i\cap U_j)\modulo q^P$. If $f\in (U_i\modulo q^P)\cap (U_j\modulo q^P)$, then
%
%\end{proof}

\begin{proof}It is clear that we can compute the intersection of $U_i\modulo(q^P)$ with $U_j\modulo(q^P)$ (see Section \ref{S3-subroutines}), but we must verify that this is equal to $(U_i\cap U_j)\modulo(q^P)$. It suffices to show that $U_i+U_j\rar K\dblparens{q}/(q^P)$ is injective.

Suppose that $f,g\in U_i+U_j$ and that $f\equiv g\modulo(q^P)$. Then $\lambda_i\lambda_jf$ and $\lambda_i\lambda_jg$ are elements of $\M_3(N,\chi^{-1};K)$ and $\lambda_i\lambda_jf\equiv \lambda_i\lambda_jg\modulo(q^P)$. Since $P\ge\Sturm_3(N)$, this congruence implies $\lambda_i\lambda_jf=\lambda_i\lambda_jg$, so $f=g$.\end{proof}

\begin{lemma}If $P\ge\Sturm_{\ell+2}(N)$ and a given subspace $W\subseteq V_{\Lambda}'(K)$ can be computed to precision $\ell P$, then $\s{S}^{T_\ell}_0(W)$ can be computed to precision $\ell P$.\end{lemma}

\begin{proof}Given a basis for $W\modulo(q^{\ell P})$ we can compute the image of this basis under $T_\ell=T_\ell^\chi$ to precision $P$ (see \ref{S3-subroutines}). Therefore, we can compute
\[\set{f\in W\modulo(q^{\ell P})}{T_\ell f\in W\modulo(q^P)}\]
to precision $q^{\ell P}$ using linear algebra on formal Laurent series. To show that the space above is equal to $\s{S}_0^{T_\ell}(W)\modulo(q^{\ell P})$ it suffices to prove that the truncated $q$-expansion map $W+T_\ell fW\rar K\dblparens{q}/(q^P)$ is injective.

Let $\lambda\in\Lambda$ and let $Q_\ell$ denote the $\ell$th multiplicative Hecke operator. If $f\in W+T_\ell W$, then $Q_\ell\lambda\cdot f\in\M_{\ell+2}(N,\chi^{-\ell};K)$. Therefore, if $f,g\in W+T_\ell W$ and $f\equiv g\modulo(q^P)$, the congruence $Q_\ell\lambda\cdot f\equiv Q_\ell\lambda\cdot g\modulo(q^P)$ implies $Q_\ell\lambda\cdot f=Q_\ell\lambda\cdot g$ since $P\ge\Sturm_{\ell+2}(N)$, whence $f=g$.\end{proof}

By induction, we obtain the following:

\begin{lemma}\label{Lemma-Precision}If $U_0,\ldots,U_s$ can all be computed to precision $\ell P$ where $P\ge\Sturm_{\ell+2}(N)$, then $V_{\Lambda,\ell}'(K)$ can also be computed to precision $\ell P$.\end{lemma}

%Combining the work of this section we obtain
%
%\begin{theorem}Fix $(N,\chi,\Lambda,\ell)$ and define $S$ as above. There is an algorithm which, on input $(\f{p},P)\in S\times\Z$, outputs the image of a basis of $V_{\Lambda,\ell}'(\kappa_\f{p})$ to precision $P$.
%
%The space $V_{\Lambda,\ell}'(\kappa_\f{p})$ is finite-dimensional, $T_\ell$-stable, and it satisfies
%\[\S_1(N,\chi;\kappa_\f{p})\subseteq V_{\Lambda,\ell}'(\kappa_\f{p})\subseteq\M_1^*(N,\chi;\kappa_\f{p})\pd\]
%Furthermore, if one of the conditions of Theorem \ref{HSThm2} are met, we have $V_{\Lambda,\ell}'(\kappa_\f{p})\subseteq\M_1(N,\chi;\kappa_\f{p})$.\end{theorem}

\subsection{Constituent computations of the Hecke stability method}\label{S3-subroutines}

Let us briefly explain how the constituent computations of the HSM above are performed using linear algebra.

If $\boldsymbol{A}$ is a matrix with entries in a subring $\s{O}\subseteq K$ we say that a prime $\f{p}\subseteq\s{O}$ is a {\it prime divisor} of $\boldsymbol{A}$ if the nullity of $\boldsymbol{A}$ (over the field $K$) increases upon reduction of the matrix modulo $\f{p}$. If $\f{p}$ is a prime divisor of $\boldsymbol{A}$, then $\f{p}$ divides the determinant of any nonsingular minor of $\boldsymbol{A}$. Therefore, in practice, to compute (a list of candidates for) the prime divisors of $\boldsymbol{A}$, we find two nonsingular minors $\boldsymbol{A}_1$ and $\boldsymbol{A}_2$, and then we factor the ideal $\gcd(\det\boldsymbol{A}_1,\det\boldsymbol{A}_2)$ of $\s{O}$.

Suppose that $W$ is a finite-dimensional subspace of $K\dblparens{q}$ where $K$ is a global field. Because there is $d$ large enough so that $q^dW\subseteq K\dblbrackets{q}$, we will assume for simplicity that $W\subseteq K\dblbrackets{q}$. Fix a basis $\{g_1,\ldots,g_r\}$ for $W$.

The table below summarizes the linear-algebraic computations performed by the Hecke stability method with input $W$ (represented by the chosen basis to an appropriate level of precision). In the table, $U$ is another finite-dimensional subspace of $K\dblbrackets{q}$ with basis $\{h_1,\ldots,h_u\}$, and $T$ is a $K$-linear operator on $K\dblbrackets{q}$. The subspace operations $\s{I}^U$ and $\s{S}^T$ were defined in Section \ref{S3-subspace-ops}.

{\small \begin{center}\begin{tabular}{|llll|}\hline target & matrix dim's & $ij$th entry of matrix & computed from matrix $\boldsymbol{A}$ by\\\hline
primes of bad&&&\\
reduction for $W$ & $r\times P$ & $j$th coefficient of $g_i$ & $\subseteq$ prime divisors of $\boldsymbol{A}$\\\hline
 & & $j$th coeff.\ of $i$th entry of & isomorphism $\ker(\boldsymbol{A})\rar\text{target}$\\
$\s{I}_\f{p}^U(W)$&$(r+u)\times P$& $(g_1,\ldots,g_r,h_1,\ldots,h_u)$&$(v_1,\ldots,v_{r+u})\mapsto\sum_{i=1}^rv_ig_i$\\\hline

divisors of & & $j$th coeff.\ of $i$th entry of &\\
$\s{I}^U$ at $W$&$(r+u)\times P$& $(g_1,\ldots,g_r,h_1,\ldots,h_u)$&$\subseteq$ prime divisors of $\boldsymbol{A}$\\\hline

%divisors of & & &\\
%$\s{I}^U$ at $W$&&& $\subseteq$ prime divisors\\\hline
 &  & $j$th coeff.\ of $i$th entry of & isomorphism $\ker(\boldsymbol{A})\rar\text{target}$\\
$\s{S}_\f{p}^T(W)$ & $2r\times P$& $(g_1,\ldots,g_r,Tg_1,\ldots,Tg_r)$&$(v_1,\ldots,v_{2r})\mapsto\sum_{i=1}^{r}v_{i+r}g_i$\\\hline

divisors of & & $j$th coeff.\ of $i$th entry of &\\
$\s{S}^T$ at $W$&$2r\times P$& $(g_1,\ldots,g_r,Tg_1,\ldots,Tg_r)$&$\subseteq$ prime divisors of $\boldsymbol{A}$\\\hline

%divisors of  & & & $\subseteq$ prime divisors of matrix\\
%$\s{S}^T$ at $W$&&& above, taking $\f{p}=0$\\\hline
\end{tabular}\end{center}}

In the first row of the table above, $P$ is taken large enough so that the matrix has rank $r$. Elsewhere, $P$ is taken to be large enough so that the given map from the kernel to the target is an isomorphism when $\f{p}=0$.

\subsection{Certification of Hecke stability hypotheses, examples and remarks}\label{S3-certification}

Given input $(N,\chi,\Lambda,\ell)$ to the Hecke stability method as outlined above, we have for each $\f{p}\in S$ a {\it Hecke stability hypothesis}: the proposition ``$V'(\kappa_\f{p})\subseteq\M_1(N,\chi;\kappa_\f{p})$.'' The truth of the Hecke stability hypothesis at $\f{p}=0$ is guaranteed by (i.) of Theorem \ref{HSThm2}, but for nonzero $\f{p}$ some work must be done to certify such a claim.

Here are four methods for certifying a Hecke stability hypothesis:

\begin{itemize}\item[a.] If the space $V'(\kappa_\f{p})$ is equal to $\Red_\f{p} V_{\Lambda,\ell}'(K)$, then reduction $\S_1(N,\chi;\s{O}_\f{p})\rar\S_1(N,\chi;\kappa_\f{p})$ is surjective and the inclusion $V'(\kappa_\f{p})\subseteq\M_1(N,\chi;\kappa_\f{p})$ holds automatically.

\item[b.] If one can prove that condition (ii.) or condition (iii.) of Theorem \ref{HSThm2} holds with $F=\kappa_\f{p}$, then $V'(\kappa_\f{p})=V'_{\Lambda,\ell}(\kappa_\f{p})\subseteq\M_1(N,\chi;\kappa_\f{p})$.

\item[c.] If one has detailed knowledge of $\mathrm{Z}(\Lambda)$ in advance, more specific arguments using polar conditions on isogeny graphs can be used to prove Hecke stability hypotheses (see Example \ref{Elliptic-Improvements} and Remark \ref{CM-Improvements} below).

\item[d.] We have $V'(\kappa_\f{p})\subseteq\M_1(N,\chi;\kappa_\f{p})$ if for some $k\ge 2$ and every $f\in V'(\kappa_\f{p})$ we have $f^k\in\M_k(N,\chi^k;\kappa_\f{p})$. Because this containment condition is ``non-linear'' it can be used to certify Hecke stability hypotheses (by checking the condition on a basis for $V'(\kappa_\f{p})$) but it cannot be used directly to compute $\M_1(N,\chi;\kappa_\f{p})$. This is an especially convenient certification method when $\chi^2=\boldsymbol{1}$, since the HSM requires that we compute a basis for $\S_2(N,\boldsymbol{1};\kappa_\f{p})$.

\end{itemize}

\begin{remark}\label{Zeros-Remark}For the second and third methods above, it is useful to have some method of computing the zeros of a modular form $\lambda\in\M_1(N,\chi^{-1};\bar{\kappa})$ with an aim towards counting its supersingular zeros. There are several ways to do this, and we outline just one below.

The principal challenge is computing the polynomial $J_\lambda(X)=\prod_{\lambda(\tau)=0}(X-j(\tau))$. Suppose that we know the $q$-expansion of the Atkin--Lehner twist $\lambda^{w}$ of $\lambda$ (if $\lambda$ is an Eisenstein series this is easy); $\lambda\lambda^w$ is a weight $2$ modular form for $\Gamma_0(N)$. Consider the modular ratio
\[H_\lambda(X)=\frac{(12G_4)^3-(G_4^3-G_6^2)X}{\lambda\lambda^w}\]
of weight $10$ for $\Gamma_0(N)$ over the polynomial ring $\kappa[X]$ (here $G_k$ is the normalized weight $k$ Eisenstein series for $\SL_2(\Z)$). Note that for almost all values of ${\alpha}\in \bar{\kappa}$, the negative part of $\div H_\lambda(\alpha)$ is $-\div(\lambda\lambda^w)$. However, if $(\lambda\lambda^w)(\tau)=0$, then the numerator of $H_\lambda(X)$ vanishes at $X=j(\tau)$ so the negative part of $\div H_\lambda(j(\tau))$ is at least $-\div(\lambda\lambda^w)+[\tau]$.

Now, for even $k$ let $M_k=\M_k(N,\boldsymbol{1};\kappa)$. If $k\ge 4$, then for any $\tau'$ we have
\[\dim\set{g\in M_k}{\div g\ge \div(\lambda\lambda^w)-[\tau']}>\dim\set{g\in M_k}{\div g\ge \div(\lambda\lambda^w)}\pd\]
The right hand side is the generic dimension of $H_\lambda(\alpha)M_k\cap M_{k+10}$ as $\alpha$ ranges over $\bar{\kappa}$, while the left hand side is a lower bound on the dimension of $H_\lambda(j(\tau))M_k\cap M_{k+10}$ when $\tau\in\mathrm{Z}(\{\lambda,\lambda^w\})$. In the language of \ref{S3-detection}, determining $J_\lambda(X)J_{\lambda^w}(X)$ reduces to finding the prime ideals of $\kappa[X]$ that divide the intersection operator $\s{I}^{H_\lambda(X)M_k}$ at $M_{k+10}$.

If $\kappa$ is a finite field, we can find these divisors by taking determinants of a matrix with entries in $\kappa[X]$ using polynomial interpolation. When $\kappa$ is a number field, one can perform interpolation over several residue fields and then reconstruct $J_\lambda(X)J_{\lambda^w}(X)$ using the Chinese remainder theorem.

\end{remark}

%To better illustrate how Hecke stability hypotheses cn be certified we will give two examples. We give a third example to 

\begin{example}\label{Elliptic-Improvements}Let $N\ge 1$ and let $\chi:(\Z/N\Z)^\times\rar\Qbar^\times$ be a character of conductor $M$ where $gX_0(M)=0$. In this case, by taking oldforms of type $(M,\chi^{-1})$, we may choose $\Lambda\subseteq\M_1(N,\chi^{-1};K)$ such that $\mathrm{Z}(\Lambda)$ contains {\it only elliptic points}. To prove that $V_{\Lambda,\ell}'(\kappa_\f{p})\subseteq\M_1(N,\chi;\kappa_\f{p})$ for a fixed nonzero prime $\f{p}$, it suffices to prove that there is no polar condition $P$ on $\s{G}_\ell(\SL_2(\Z);\kappa_\f{p})$ satisfying $P\subseteq {W}_1$.

For concreteness, suppose that $j\mathrm{Z}(\lambda)=\{0\}$ (i.e., $M=3$ or $M=7$) and that $\ell=2$ (so since $\ell\nmid N$, $N$ must be odd). We have $\Phi_2(0,Y)=(Y-54000)^3$ and
\[\Phi_2(54000,Y)=Y(Y^2-2835810000Y+6549518250000)\]
where $\Phi_2(X,Y)$ is the second modular polynomial [ref]. When $p\nmid 6549518250000=2^4 \cdot 3^9 \cdot 5^6 \cdot 11^3$, there is a {\it unique} $2$-isogeny from the elliptic curve $E'$ with $j$-invariant $54000$ (mod $p$) to the elliptic curve $E$ with $j$-invariant $0$ (mod $p$). For such $p$, any polar condition on $\s{G}_2(\Gamma(1);\kappa_\f{p})$ that contains the vertex $[E]$ must also contain the vertex $[E']$ (which is distinct from $[E]$). Since $\delta_{N,1}\mathrm{Z}(\Lambda)=\{[E]\}$, it follows that $V_{\Lambda,2}'(\kappa_\f{p})\subseteq\M_1(N,\chi;\kappa_\f{p})$ as long as $\f{p}\cap\Z\notin\{2,3,5,11\}$.

In fact, we have proven that $\set{f\in V_{\Lambda}'(\kappa)}{T_2f\in V_{\Lambda}'(\kappa)}\subseteq \M_1(N,\chi;\kappa)$ in this situation.

\end{example}

\begin{example}Let $(N,\chi)=(651,\varepsilon)$ where $\varepsilon$ is the quadratic character of level $651$ and conductor $31$. Applying the Hecke stability method in characteristic zero with $\Lambda=\{\lambda\}$ where
\[\lambda(q)=3+2\sum_{n=1}^\infty\sum_{d\mid n}\left(\frac{d}{31}\right)q^n\]
proves that $\S_1(651,\varepsilon;\Q)$ is trivial. However, the methods of Section \ref{S3-detection} indicate that mod $337$ reduction $\S_1(651,\varepsilon;\Z[\tfrac{1}{651}])\rar\S_1(651,\varepsilon;\F_{337})$
may not be surjective. Indeed, $V_{\Lambda,2}'(\F_{337})$ is $2$-dimensional. We can verify the Hecke stability hypothesis $V_{\Lambda,2}'(\F_{337})\subseteq\M_1(N,{\varepsilon};\F_{337})$ in three ways:

\begin{itemize}\item Using the method of Remark \ref{Zeros-Remark}, we find that $\lambda(\tau)=0$ only if $j(\tau)$ is a root of
\[3^{32}X^3+394086965048982896640X^2+23574729187315780314726400X\cm\]
which factors as $2X(X-96)(X-241)$ when reduced modulo $337$. The $j$-invariants $0$, $96$, and $241$ are all ordinary over $\F_{337}$, so the Hecke stability hypothesis holds by condition (ii.)\ of Theorem \ref{HSThm2}.

\item Applying (iii.)\ of Theorem \ref{HSThm2} with $M=31$, $\ell=2$, and $r=2$. Here, the inequality $337>\max\{\frac{1024}{31},64\}=64$ guarantees that any polar condition on $\s{G}_2(\Gamma_0(31);\F_{337})$ contains at least $3$ nonelliptic vertices, but $gX_0(31)=2$, so $\lambda$ has at most $2$ nonelliptic zeros on $X_0(31)$. Therefore, $\mathrm{Z}(\Lambda)$ cannot contain a polar condition, so elements of $V'_{\Lambda,2}(\F_{337})$ must be modular forms.

\item Let $\{f_1,f_2\}$ be a basis for $V_{\Lambda,2}(\F_{337})$. To check the Hecke stability hypothesis, it is enough to verify that $f_1^2,f_2^2\in\M_2(651,\boldsymbol{1};\F_{337})$, and this is the case.\end{itemize}

Using any of these methods, we find that $\dim \S_1(651,\varepsilon;\F_{337})=2$. The discrepancy in dimensions between this space over $\Q$ and $\F_{337}$ indicates the existence of nontrivial $337$-torsion in the cohomology $\H^1(X_1(651),\underline{\omega}(-\mathrm{cusps}))$ [ref].\end{example}

\begin{remark}\label{CM-Improvements}Though we will not go into the details, some improvements to the bounds in (iii.)\ of Theorem \ref{HSThm2} can be formulated using the Goren--Lauter lemma when $\delta_{N,M}\mathrm{Z}(\Lambda)$ is known to consist of CM points on $X_0(M)$.\end{remark}

\begin{remark}The author has so far encountered only one family of {\it false} Hecke stability hypotheses: If $\varepsilon$ is the quadratic character of conductor $11$ and $\lambda$ is the unique normalized weight $1$ Eisenstein series of type $(11,\varepsilon)$, then $\S_1(11,\varepsilon;\F_7)$ is trivial, but
\[\dim_{\F_7}V_{\{\lambda\},\ell}'(\F_7)=1\]
for $\ell\in\{5,59,\ldots\}$. Note that in this case, the zeros of $\lambda$ (which lie over $j=-32^3$, the elliptic curve $\C/\Z[\frac{1+\sqrt{-11}}{2}]$) are supersingular, and the inequalities of (iii.)\ in Theorem \ref{HSThm2} fail to obtain.\end{remark}

\begin{proof}[Note] This section concludes the proof of Theorem \ref{DetectionThm}.\end{proof}

%
%
%The space $\S_2(11,\boldsymbol{1};\F_7)$ is $1$-dimensional, generated by the mod $7$ reduction $g$ of the cusp form attached to the elliptic curve $y^2+y=x^3-x^2-10x-20$. Let $\varepsilon$ be the mod $11$ quadratic character. $\S_1(11,\varepsilon;\F_7)$ is trivial, but if one takes $\lambda$ to be the unique normalized Eisenstein series of type $(11,\varepsilon)$, then $g/\lambda$ is a $T_5$-eigenvector: $T_5(g/\lambda)=5g/\lambda$ (in $\F_7$).
%
%Here is an explanation: Even though $g/\lambda$ is not a modular form, all of its poles are supersingular (by Theorem \ref{HSThm1}.b), and $A\cdot g/\lambda$ is an element of $\S_7(11,\varepsilon;\F_7)$ where $A$ is the Hasse invariant. The space $\S_7(11,\varepsilon;\F_7)$ decomposes as a direct sum $M_1\oplus M_2$ of simple Hecke modules, and $A\cdot g/\lambda\in M_1$ where $M_1$ is $4$-dimensional; $A\cdot g/\lambda$ can be obtained as the {\it trace} of (a scalar multiple of) a newform $f\in M_1$ (whose $q$-expansion is in $\F_{7^4}$). The primes $\ell$ such that $A\cdot g/\lambda$ is a $T_\ell$-eigenform are those primes $\ell$ for which $a_\ell(f)\in M_1$; $\ell=5$ is the smallest such prime, and the next is $\ell=59$.\end{example}

\subsection{An informal discussion of complexity}\label{S3-complexity}

We refrain from a detailed account of the complexity of the Hecke stability method, since its complexity depends on the efficiency of other algorithms already in place.

%Though many of the operations involved are ``just linear algebra'' over a number field (determined by the character $\chi$), some of the other computations are more nontrivial from an algorithmic standpoint.

\begin{itemize}\item The HSM requires that we compute a basis for $\S_2(N,\boldsymbol{1};\Z[\frac{1}{N}])$ to high precision: For fixed $\ell$ (which can be assumed to be the smallest prime not dividing $N$) Lemma \ref{Lemma-Precision} requires that we compute a basis for the weight $2$ cusp forms to precision $\ell\Sturm_{\ell+2}(N)=\mathrm{O}(\ell^2 N)$.

Fixing $\ell$, the complexity of computing this space using modular symbols is in $\mathrm{O}(N^{4+\epsilon})$, though a precise reference for this is difficult to find. If $N$ is squarefree, there are GRH-conditional results due to Bruin that reduce the computational complexity to polynomial in $\log N$ \cite{Bruin} (see also \cite{EdixBook}).

\item  Implementing the algorithm of Theorem \ref{DetectionThm} requires that we evaluate the determinants of large matrices over a number field $K$ where $[K:\Q]=\ord\chi$. Fast algorithms for computing the determinant of an $r\times r$ matrix require about $\mathrm{O}(r^3)$ field operations (e.g., Gaussian elimination), but over a number field, these operations are essentially operations on {\it polynomials} of degree $[K:\Q]-1$. The time and memory requirements for the na\"ive adaptation of these methods to $K$ scale poorly as the degree increases.

Some computational tricks using the Chinese remainder theorem and the theory of cyclotomic fields yield noticeable improvements to computing such determinants in practice.

\item Once the determinants of the matrices above are computed, we compute their GCD $\f{a}$ and then factor $\f{a}$ to determine candidates for nonsurjectivity (see Section \ref{S3-subroutines}). Without an estimate of the size of $\f{a}$, it is not clear how difficult this factorization problem is, but it is a necessary step if one wants to completely characterize weight $1$ forms of a given type. Our data suggest that $\mathrm{N}(\f{a})$ grows quickly in the index of $\Gamma_1(N)$ (see Section \ref{S4-growth}).

\end{itemize}

\section{Products of the Hecke stability method}\label{S4}

The Hecke stability method provides us with a systematic method for computing tables of weight $1$ modular forms. Because of the correspondence between mod $p$ modular forms and Galois representations with controlled ramification, these data are relevant to the refined inverse Galois problem for $\PXL_2(\F_q)$. In what follows we comment on some of our more remarkable findings.

All tables produced by the Hecke stability method are currently being integrated into ``The $\mathrm{L}$-functions and modular forms database'' (LMFDB). The known data are also available by request, and some tables are also provided in \cite{Dissertation}.

%\subsection{}

\subsection{Growth of torsion in $\H^1(X_1(N),\underline{\omega}(-\mathrm{cusps}))$}\label{S4-growth}

Extensive computations with the Hecke stability method have produced evidence that the torsion subgroup of $\H^1(X_1(N),\underline{\omega}(-\mathrm{cusps}))$ grows {\it at least} exponentially in the index of $\Gamma_1(N)$.

\begin{conjecture}The limit
\[\liminf_{N\rar\infty}\left(\frac{\log|\H^1(X_1(N),\underline{\omega}(-\mathrm{cusps}))_{\mathrm{tors}}|}{[\SL_2(\Z):\Gamma_1(N)]}\right)\] is nonzero.\end{conjecture}

%In fact, there is reason to believe that asymptotic may be much stronger.
%[Akshay citationx]

The most coherent set of data relevant to the conjecture above is given by the contributions to torsion from $\S_1(N,\varepsilon;\Fbar{p})$ where $N=3\nu$ for a prime $\nu$ and $\varepsilon$ is the quadratic character of conductor $3$. These data are relatively easy to obtain because the forms are defined over $\Z$, one can take $s=1$ and $t=1$ (as defined in Section \ref{S3-HSM}), and the relevant Hecke stability hypotheses over $F$ require no extra certification step when $\mathrm{char}(F)\notin\{2,3,5,11\}$ (see Example \ref{Elliptic-Improvements}).

For $N$ divisible by $3$, set
\[t'(N,\varepsilon)=\prod_{p\nmid N\phi(N)}p^{\dim\S_1(N,\varepsilon;\F_p)-\dim\S_1(N,\varepsilon;\Q)}\pd\]
The exclusion of primes dividing $\phi(N)$ guarantees the equalities
\begin{align*}\dim\S_1(N,\varepsilon;\F_p)-\dim\S_1(N,\varepsilon;\Q)&=\dim\M_1(N,\varepsilon;\F_p)-\dim\M_1(N,\varepsilon;\Q)\\&=|\S_1(N,\varepsilon;\Fbar{p})^{\mathrm{new}}|-|\S_1(N,\varepsilon;\Qbar)^{\mathrm{new}}|\cm\end{align*}
by avoiding certain pathologies (e.g., Hecke modules in characteristics dividing $\phi(N)$ do not always admit bases consisting of Hecke eigenforms). Because the primes dividing $t'(N,\varepsilon)$ grow precipitously, excluding these small primes does not affect the magnitude of $t'(N,\varepsilon)$ overmuch.

%\begin{example}Here is the \end{example}

%that the modular ratios of our initial space $U_0$ vanish at every cusp (mod $p$) by avoiding Eisenstein congruences of type $(2,N,\boldsymbol{1})$. Because $t'(N,\varepsilon)$ is a lower bound for contribution to torsion from forms of a single nebentypus, the quantity we are concerned with replaces the index of $\Gamma_1(N)$ in the ratio of Conjecture x with the index of $\Gamma_0(N)$.
%
%When $p\nmid\phi(N)$, $\dim\S_1(N,\varepsilon;\F_p)-\dim\S_1(N,\varepsilon;\Q)$ is equal to the number of newforms of type $(1,N,\varepsilon)$ over $\Fpbar$ that do not lift to characteristic zero. When $p\mid\phi(N)$ one may see Eisenstein congruences mod $p$ or $\S_1(N,\varepsilon;\Fbar{p})$ may not have a basis of Hecke eigenforms.

Below we give a table for $t'(N,\varepsilon)$ where $N=3\nu$ and $\nu$ ranges over all primes in $[47,600]$. There are three such levels where $t'(N,\varepsilon)=1$, namely $177$, $183$, and $201$---the table skips these levels. For every level in the table, $\dim_\Q\S_1(N,\varepsilon;\Q)=0$

%, consistent with the phenomenon observed in Section \ref{S4-repulsion}.

Finally, it should be mentioned that the Hecke stability hypotheses over $\F_5$ and $\F_{11}$ were certified using method (d.)\ of Section \ref{S3-certification} (all other Hecke stability hypotheses follow from the argument in Example \ref{Elliptic-Improvements}).

\begin{center}\noindent{\footnotesize\begin{tabular}{|rlc|rlc|}\hline $N$ & $t'(N,\varepsilon)$ & {\tiny$\frac{\log t'(N,\varepsilon)}{[\SL_2(\Z):\Gamma_0(N)]}$} & $N$ & $t'(N,\varepsilon)$ & $\frac{\log t'(N,\varepsilon)}{[\SL_2(\Z):\Gamma_0(N)]}$
\\\hline$141$ & $5^2$ & $0.016765$ &  $933$ & $23^2 \cdot 11177^2 \cdot 2036539^2$ & $0.043243$\\
$159$ & $5^2$ & $0.014902$ &  $939$ & $11643684611137^2$ & $0.047907$\\
$213$ & $17^2$ & $0.019675$ &  $951$ & $13^2 \cdot 593^2 \cdot 1427^2 \cdot 1363631^2$ & $0.047703$\\
$219$ & $41^2$ & $0.025092$ &  $993$ & $1399^2 \cdot 5767763143^2$ & $0.044758$\\
$237$ & $5^2$ & $0.010059$ &  $1011$ & $47^2 \cdot 5879^2 \cdot 6004682531^2$ & $0.051842$\\
$249$ & $89^2$ & $0.026718$ &  $1041$ & $5^2 \cdot 297255711552329^2$ & $0.050194$\\
$267$ & $41^2$ & $0.020631$ &  $1047$ & $568492143207481^2$ & $0.048534$\\
$291$ & $73^2$ & $0.021890$ &  $1059$ & $293^2 \cdot 411688918531337^2$ & $0.055553$\\
$303$ & $199^2$ & $0.025948$ &  $1077$ & $61^2 \cdot 1097549^2 \cdot 17174569^2$ & $0.048164$\\
$309$ & $11^2$ & $0.011528$ &  $1101$ & $33377^2 \cdot 167296128361^2$ & $0.049264$\\
$321$ & $73^2$ & $0.019863$ &  $1119$ & $7^2 \cdot 13^2 \cdot 101^2 \cdot 1567^2 \cdot 30449^2 \cdot 75629^2$ & $0.050856$\\
$327$ & $281^2$ & $0.025629$ &  $1137$ & $5^2 \cdot 11^2 \cdot 1811^2 \cdot 7757^2 \cdot 18835005559^2$ & $0.058058$\\
$339$ & $11801^2$ & $0.041123$ &  $1149$ & $17^2 \cdot 587^2 \cdot 7459050493709^2$ & $0.050584$\\
$381$ & $13^2$ & $0.010019$ &  $1167$ & $19^2 \cdot 43531931^2 \cdot 276464041661^2$ & $0.060101$\\
$393$ & $7^2 \cdot 1669^2$ & $0.035477$ &  $1191$ & $23^2 \cdot 8219^2 \cdot 264610601669^2$ & $0.048306$\\
$411$ & $223^2 \cdot 613^2$ & $0.042846$ &  $1203$ & $17^2 \cdot 566759^2 \cdot 83439102139^2$ & $0.051279$\\
$417$ & $37^2 \cdot 227^2$ & $0.032271$ &  $1227$ & $5^2 \cdot 17^6 \cdot 23^2 \cdot 4501099^2 \cdot 106521127^2$ & $0.057376$\\
$447$ & $353^2 \cdot 937^2$ & $0.042364$ &  $1257$ & $90379^2 \cdot 6166483^2 \cdot 27175307^2$ & $0.052576$\\
$453$ & $2417^2$ & $0.025626$ &  $1263$ & $11^4 \cdot 167^4 \cdot 3359^2 \cdot 589072513^2$ & $0.051357$\\
$471$ & $4523^2$ & $0.026636$ &  $1293$ & $67^2 \cdot 19215219958141279^2$ & $0.048263$\\
$489$ & $11^2 \cdot 463^2$ & $0.026023$ &  $1299$ & $17^2 \cdot 116866832706907338779^2$ & $0.056499$\\
$501$ & $191^2 \cdot 859^2$ & $0.035738$ &  $1317$ & $5^2 \cdot 66025499784624707^2$ & $0.045839$\\
$519$ & $257^2 \cdot 523^2$ & $0.033933$ &  $1329$ & $2316157573^2 \cdot 743496822373^2$ & $0.055065$\\
$537$ & $5^2 \cdot 97^2 \cdot 8713^2$ & $0.042380$ &  $1347$ & $2311^2 \cdot 6717077^2 \cdot 1707229020033067^2$ & $0.065044$\\
$543$ & $67^2 \cdot 193^2$ & $0.026009$ &  $1371$ & $227^2 \cdot 751^2 \cdot 455999^2 \cdot 1224778807207^2$ & $0.057763$\\
$573$ & $59^2 \cdot 397^2 \cdot 439^2$ & $0.042047$ &  $1383$ & $47^2 \cdot 3108109^2 \cdot 8284637^2 \cdot 129033803^2$ & $0.057798$\\
$579$ & $67^2 \cdot 15731^2$ & $0.035743$ &  $1389$ & $23^2 \cdot 71166715339905984791^2$ & $0.052637$\\
$591$ & $29^2 \cdot 444151^2$ & $0.041341$ &  $1401$ & $5^2 \cdot 857030122343348247418273^2$ & $0.060595$\\
$597$ & $19^2 \cdot 62617^2$ & $0.034973$ &  $1437$ & $7^2 \cdot 20753^2 \cdot 3413228933067061^2$ & $0.049638$\\
$633$ & $451933^2$ & $0.030711$ &  $1461$ & $1561303^2 \cdot 68396038855089373^2$ & $0.054329$\\
$669$ & $11489^2 \cdot 48883^2$ & $0.044969$ &  $1473$ & $113^2 \cdot 525503753^2 \cdot 981045960949^2$ & $0.053272$\\
$681$ & $163316303^2$ & $0.041472$ &  $1497$ & $347095820483556600660647^2$ & $0.054204$\\
$687$ & $1140121^2$ & $0.030319$ &  $1509$ & $1259^2 \cdot 7689582211^2 \cdot 18300110183^2$ & $0.053107$\\
$699$ & $41^2 \cdot 89^2 \cdot 223^2 \cdot 1109^2$ & $0.044061$ &  $1527$ & $11^2 \cdot 732709^2 \cdot 296442544991039301013^2$ & $0.061805$\\
$717$ & $619^2 \cdot 26921^2$ & $0.034643$ &  $1563$ & $139^2 \cdot 132467773^2 \cdot 189430903^2 \cdot 2365816319^2$ & $0.061571$\\
$723$ & $770297621^2$ & $0.042277$ &  $1569$ & $53^2 \cdot 62196526144813901269147^2$ & $0.053869$\\
$753$ & $23741^2 \cdot 13669147^2$ & $0.052590$ &  $1623$ & $19^2 \cdot 14280291317898094529842523^2$ & $0.056149$\\
$771$ & $97^2 \cdot 3833^2 \cdot 11383^2$ & $0.042957$ &  $1641$ & $53^2 \cdot 73^2 \cdot 1489^2 \cdot 8052073^2 \cdot 36978604703^2$ & $0.050914$\\
$789$ & $23^2 \cdot 101879593^2$ & $0.040861$ &  $1671$ & $7^4 \cdot 71^2 \cdot 2198029^2 \cdot 5826349757405731^2$ & $0.052920$\\
$807$ & $20983^2 \cdot 887059^2$ & $0.043791$ &  $1689$ & $11^2 \cdot 68581518742288026772115454991^2$ & $0.060989$\\
$813$ & $53^2 \cdot 71^2 \cdot 51803^2$ & $0.035089$ &  $1707$ & $5^2 \cdot 277^2 \cdot 154664726857^2 \cdot 1697802067421853871^2$ & $0.065767$\\
$831$ & $13^2 \cdot 59^2 \cdot 79^2 \cdot 311603^2$ & $0.042556$ &  $1713$ & $5^2 \cdot 107^2 \cdot 70663^2 \cdot 101011208101^2 \cdot 127452545273^2$ & $0.059753$\\
$843$ & $709^2 \cdot 1471^2 \cdot 165606751^2$ & $0.058125$ &  $1731$ & $31^2 \cdot 433^2 \cdot 19159873^2 \cdot 36001051344486600557^2$ & $0.061681$\\
$849$ & $421^2 \cdot 4325322613^2$ & $0.049701$ &  $1761$ & $5^2 \cdot 11^2 \cdot 5212393^2 \cdot 121756570064236471668019^2$ & $0.061760$\\
$879$ & $19^2 \cdot 139^2 \cdot 13339^2$ & $0.029553$ &  $1779$ & $1104047325433567^2 \cdot 615184180444752239^2$ & $0.063635$\\
$921$ & $166017730847^2$ & $0.041941$ &  $1797$ & $5^2 \cdot 11^2 \cdot 755809^2 \cdot 14302275115816198137271^2$ & $0.057131$\\\hline\end{tabular}}\end{center}

\hspace{-1cm}\includegraphics[scale=1]{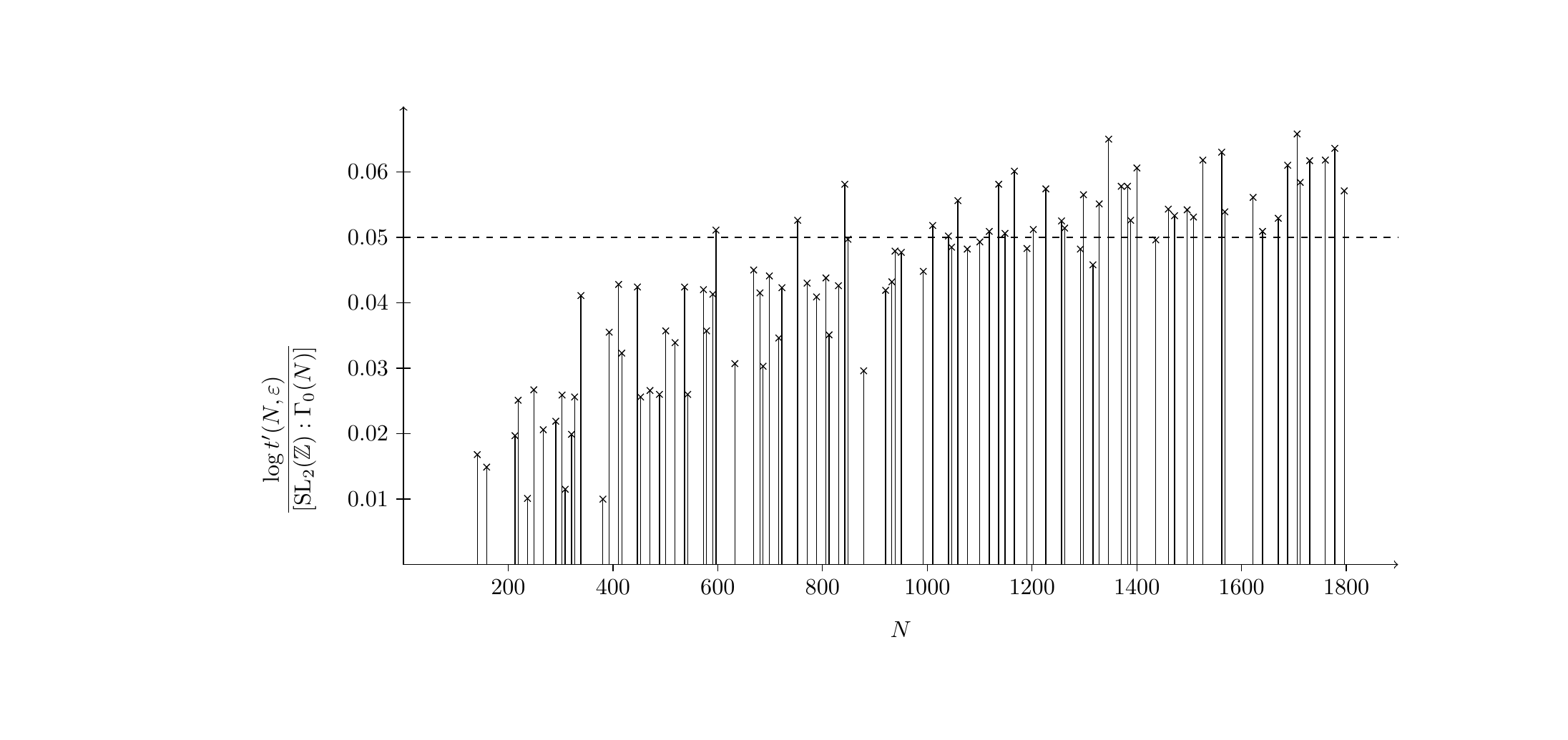}

The reader may observe that $t'(N,\varepsilon)$ is always a square in the table above. This is because the Atkin--Lehner twist is an involution on $\S_1(N,\varepsilon;\Fbar{p})$ whose fixed points are dihedral newforms of type $(N,\varepsilon)$. Because dihedral newforms always lift to characteristic zero (in the same level and character) \cite{Wiese1}, $\dim\S_1(N,\varepsilon;\Fbar{p})-\dim\S_1(N,\varepsilon;\Qbar)$ is always even. This argument generalizes:

\begin{theorem}If $\chi$ is a quadratic character of level $N$ and $p\nmid\phi(N)$, then $\dim_{\F_p}\S_1(N,\chi;\F_p)\equiv\dim_{\Q}\S_1(N,\chi;\Q)\modulo 2$.\end{theorem}

\subsection{Galois number fields with small root discriminant}\label{S4-roberts}

Recall that when $K$ is a number field, the quantity $\mathrm{rd}(K)=|\mathrm{disc}(K)|^{[K:\Q]}$ is called the {\it root discriminant} of $K$. Under the generalized Riemann hypothesis,
\[\lim_{n\rar\infty}\inf\set{\mathrm{rd}(K)}{K/\Q}=8\pi e^\gamma\approx 44.76323\pd\]
The constant $8\pi e^\gamma$ is sometimes referred to as the {\it Odlyzko--Serre bound} \cite{SerreOdlyzko}.

Let $\s{E}(G)$ be the set of all Galois number fields $K\subset\C$ such that $\mathrm{rd}(K)\le 8\pi e^\gamma$ and $\Gal(K/\Q)\simeq G$. It is a natural problem to determine each of the finite sets $\s{E}(G)$.
\begin{itemize}\item Jones and Roberts have shown that $\left|\bigcup_{\text{$A$ ab.}}\s{E}(A)\right|=7063$. This led them to conjecture that the set $\bigcup_G\s{E}(G)$ is finite \cite{JonesRoberts}.
\item In her thesis \cite{Wallington}, Wallington developed algorithms for analyzing $\s{E}(G)$ with $G\le\mathrm{S}_6$. Subsequently, Jones and Wallington characterized $\s{E}(G)$ for all solvable transitive $G\le\mathrm{S}_n$ with $n\le 10$, all groups $G$ of order $\le 24$, and all dihedral groups of prime degree \cite{WallingtonJones}.
\item {\it Nonsolvable} $K$ with $\mathrm{rd}(K)\le 8\pi e^\gamma$ have been discovered at a rate of about two per year \cite{Roberts}.
\end{itemize}

If $f\in\S_1(N,\chi;\Fbar{p})$, then $\mathrm{rd}(K_f)=(N/\f{f}_\chi)^{(p-1)/p}\f{f}_\chi^{1/2}$ where $K_f$ is the number field constructed from the Galois representation associated to $f$ and $\f{f}_\chi$ is the conductor of $\chi$. Bearing this formula in mind, the tables of weight $1$ modular forms mod $p$ produced by the Hecke stability method (see Appendix A of \cite{Dissertation}) reveal six new examples of $\PXL_2(\F_7)$- and $\PXL_2(\F_{11})$-extensions of $\Q$ with root discriminant $\le 8\pi e^\gamma$:

%
%In practice, Galois number fields $K$ satisfying $\mathrm{rd}(K)\le\Omega$ are rare. Jones and Roberts have shown that there are exactly $7063$ Abelian number fields satisfying this condition, and (...). Nonsolvable examples have been discovered at a rate of around two per year.
%\[\s{E}=\set{K\le\C}{\text{$K/\Q$ Galois and $|D_K|^{1/[K:\Q]}<8\pi e^\gamma$}}\]
%is actually finite. Thus, it is natural to determine for each finite group $G$, the finite set $\s{E}(G)$ of all number fields $K\in\s{E}$ such that $\Gal(K/\Q)\simeq G$.

%The tables of are relevant to the problem of 
%
%
%have revealed six previously unknown number fields in $\s{E}(G)$ where $G=\PXL_2(\F_p)$ and $p\in\{7,11\}$: If $f\in\S_1(N,\chi;\Fbar{p})$ is a newform with $N$ squarefree, then the root discriminant of the associated number field $K_f$ (as constructed in the introduction) is given by $\mathrm{rd}(K_f)=(N/\f{f}_\chi)^{(p-1)/p}\f{f}_\chi^{1/2}$. Inspection of the tables yields a number of candidates, and with some work (to verify the Galois groups) we 
\begin{center}\begin{tabular}{|lrllrc|}\hline$N$ & $\f{f}_\chi$ & rt.\ disc. & $G$ & $d$ & polynomial\\\hline
$489$ & $163$  & $32.7382$ & $\PGL_2(\F_7)$ & $-163$ & $x^8-2x^7+x^6+x^5+7x^4-14x^3+7x^2+7x+1$\\

$561$ & $187$ & $35.0656$ & $\PGL_2(\F_7)$ & $17$ & $x^8-3x^7+9x^6-21x^5+42x^4-45x^3+57x^2-75x+36$\\

$705$ & $235$ & $39.3093$ & $\PGL_2(\F_7)$ & $-235$ & $x^8-2x^7-2x^6+10x^5+7x^4-11x^3-11x^2-2x+1$ \\

$341$ & $31$ & $43.4814$ & $\PSL_2(\F_7)$ & & $x^8-3x^7-2x^6+12x^5-36x^3+62x^2-51x+18$\\

$615$ & $123$ & $44.0626$ & $\PGL_2(\F_7)$ & $-123$ & 
$x^8+x^7-5x^6-10x^5+20x^4-15x^3-105x^2-275x+15$\\\hline

$681$ & $227$ & $40.9034$ & $\PGL_2(\F_{11})$ & $-227$ & \\\hline\end{tabular}\end{center}
In the above table, $d$ is the discriminant of the quadratic subfield. The polynomials in this table were kindly provided by D.\ Roberts and the corresponding number fields have been recorded in the Number Fields database \cite{Jones}.

\subsection{Nonsolvable Galois number fields ramified at a single prime}\label{S4-gross}
	
For a finite group $G$ and a set $S$ of rational primes, let $\s{K}(G,S)$ be the set of $G$-extensions of $\Q$ unramified at primes outside of $S$.

In the 1970s, Serre proved a ``large image theorem'' for the mod $p$ Galois representations attached to level $1$ newforms. As a corollary, he showed that for every prime $\nu\ge 11$ there is a nonsolvable group $G$ such that $\s{K}(G,\{\nu\})$ is nonempty. That is, for every prime $\nu\ge 11$ there is a nonsolvable number field ramified only at $\nu$.

Gross observed in the 1990s that there were no known number fields with this property for $\nu\le 7$. Examples were subsequently found by Demb\'el\'e for $\nu=2$ \cite{Dembele} (using Hilbert modular forms), Demb\'el\'e--Greenberg--Voight for $\nu=3,5$ \cite{DGV} (using Hilbert modular forms), and Dieulefait for $\nu=7$ \cite{Dieulefait} (using Siegel modular forms).

We can give a novel solution to Gross' problem using weight $1$ modular forms over $\Fbar{p}$. The table below summarizes our solution: Using the HSM we found for each $N$ below a newform $f\in\S_1(N,\chi;\Fbar{p})$ such that the number field $K_f/\Q$ is ramified only at the single prime $\nu\mid N$ and $\Gal(K_f/\Q)\simeq G$.

{\begin{center}\begin{tabular}{|llllll|}\hline
$\nu$&$N$&$p$&$\mathrm{ord}(\chi)$&$\chi(g)$&$G$
\\\hline $2$& $256$ & $374377637683781311$& $64$ &${}^*244174677499476933$&$\mathrm{PGL}_2(\F_p)$\\
%&&&$5\mapsto 244174677499476933$\\
%&&&$+304712536046995623\sqrt{3}$\\
$3$&$243$ & $44056205122990214764331$ & $162$ & ${}^*2893275249056948748386$&$\PSL_2(\F_p)$\\
$5$&$125$ &$199$& $100$ &${}^*79$&$\PSL_2(\F_p)$\\
%X^2+120X+1
$7$&$343$ & $74873$ & $98$ & $16423$ & $\mathrm{PGL}_2(\F_p)$\\\hline
\end{tabular}\end{center}}\bigskip

The $\chi(g)$ column is meant to provide enough information to retrieve the (odd) nebentypus $\chi:(\Z/N\Z)^\times\rar\Fbar{p}^\times$ for each of these forms: In the first row, $g=5$ (determining the character up to parity), and in all other rows, $g$ is the least primitive root for $\Z/N\Z$. An entry of the form $a$ means that $\chi(g)=a$, and an entry of the form ${}^*a$ means that $\chi(g)$ is an element of $\F_{p^2}$ whose trace is equal to $a$ and whose norm is $1$.

The corresponding Hecke stability hypotheses can all be certified using condition (ii.) or (iii.) of Theorem \ref{HSThm2}. In each case, we took $\Lambda=\{\lambda_{\chi^{-1}}\}$ where $\lambda_{\chi^{-1}}$ is the unique normalized weight $1$ Eisenstein series of character $\chi^{-1}$ and we took $\ell$ to be the least prime not dividing $N$.

\begin{itemize}\item When $N=256,243$ the elementary bound $|\delta_{N,N}\mathrm{Z}(\lambda_{\chi^{-1}})|\le gX_0(N)$ suffices to apply condition (iii.) of Theorem \ref{HSThm2} since $p$ is rather large compared to $N$ in these cases.

\item For $N=125,343$, the set $\delta_{N,1}\mathrm{Z}(\lambda_{\chi^{-1}})$ was computed directly using the method of Remark \ref{Zeros-Remark} over $\Fbar{p}$; in both cases, $\lambda_{\chi^{-1}}$ was found to have no supersingular non-elliptic zeros.\end{itemize}

Once newforms of these types were computed to sufficient precision, the Galois groups $G$ were identified in a manner similar to that found in \cite{Buz12}.

%\subsection{Hecke stability in other settings}{S4-extensions}
%
%One advantage of the Hecke stability method is that the basic idea is not limited to the computation of weight $1$ Katz modular forms. 

%%%%%%%%%%%%%%%%%%
%%%%%%%%%%%%%%%%%%
%%%%%%%%%%%%%%%%%%
%%%%%%%%%%%%%%%%%%
%%%%%%%%%%%%%%%%%%
%%%%%%%%%%%%%%%%%%
%%%%%%%%%%%%%%%%%%
%%%%%%%%%%%%%%%%%%
%%%%%%%%%%%%%%%%%%
%%%%%%%%%%%%%%%%%%
%%%%%%%%%%%%%%%%%%
%%%%%%%%%%%%%%%%%%
%%%%%%%%%%%%%%%%%%
%%%%%%%%%%%%%%%%%%
%%%%%%%%%%%%%%%%%%
%%%%%%%%%%%%%%%%%%
%%%%%%%%%%%%%%%%%%
%%%%%%%%%%%%%%%%%%
%%%%%%%%%%%%%%%%%%
%%%%%%%%%%%%%%%%%%
%%%%%%%%%%%%%%%%%%
%%%%%%%%%%%%%%%%%%
%%%%%%%%%%%%%%%%%%
%%%%%%%%%%%%%%%%%%
%%%%%%%%%%%%%%%%%%

\end{document}